\documentclass[11pt,reqno,tbtags]{amsart}
\pdfoutput=1
\usepackage[]{amsmath,amssymb,amsfonts,latexsym,amsthm,enumerate}
\usepackage[numeric,initials,nobysame]{amsrefs}
\usepackage{graphicx,subfigure}
\usepackage[]{bbm}

\numberwithin{equation}{section}

\newtheorem{maintheorem}{Theorem}
\newtheorem{maincoro}[maintheorem]{Corollary}
\newtheorem{mainconj}[maintheorem]{Conjecture}
\newtheorem*{conjecture*}{Conjecture}
\newtheorem{theorem}{Theorem}[section]
\newtheorem*{theorem*}{Theorem}
\newtheorem{lemma}[theorem]{Lemma}
\newtheorem{claim}[theorem]{Claim}

\newtheorem{proposition}[theorem]{Proposition}
\newtheorem{corollary}[theorem]{Corollary}

\newtheorem*{definition*}{Definition}

\theoremstyle{definition}{

\newtheorem*{remark*}{Remark}
}

\newcommand{\R}{\mathbb R}

\newcommand{\Z}{\mathbb Z}

\newcommand{\E}{\mathbb{E}}
\renewcommand{\P}{\mathbb{P}}
\DeclareMathOperator{\var}{Var}

\renewcommand{\epsilon}{\varepsilon}

\newcommand{\one}{\mathbbm{1}}
\newcommand{\cD}{\mathcal{D}}
\newcommand{\cL}{\mathcal{L}}
\newcommand{\cF}{\mathcal{F}}

\newcommand{\given}{\, \big| \,}
\newcommand{\tr}{\textsc{t}}
\newcommand{\allone}{\boldsymbol{1}}
\newcommand{\vol}{\operatorname{vol}}
\date{}

\begin{document}
\title{Spectra of lifted Ramanujan graphs}

\author{Eyal Lubetzky}
\address{Eyal Lubetzky\hfill\break
Microsoft Research\\
One Microsoft Way\\
Redmond, WA 98052, USA.}
\email{eyal@microsoft.com}
\urladdr{}

\author{Benny Sudakov}
\address{Benny Sudakov\hfill\break
Department of Mathematics\\
UCLA\\
Los Angeles, CA 90095, USA.}
\email{bsudakov@math.ucla.edu}
\urladdr{}
\thanks{\hspace{-0.25cm}\makebox[\textwidth][l]{
B.\ Sudakov is supported by NSF CAREER award 0812005 and a USA-Israeli BSF grant.}}

\author{Van Vu}
\address{Van Vu\hfill\break
Department of Mathematics\\
Rutgers\\
Piscataway, NJ 08854, USA.}
\email{vanvu@math.rutgers.edu}
\urladdr{}
\thanks{\hspace{-0.25cm}
V.\ Vu is supported by research grants DMS-0901216 and AFOSAR-FA-9550-09-1-0167.}

\begin{abstract}
A random $n$-lift of a base graph $G$ is its cover graph $H$ on the vertices $[n]\times V(G)$, where for each edge $u v$ in $G$ there is an independent uniform bijection $\pi$, and $H$ has all edges of the form $(i,u),(\pi(i),v)$. A main motivation for studying lifts is understanding Ramanujan graphs, and namely whether typical covers of such a graph are also Ramanujan.

Let $G$ be a graph with largest eigenvalue $\lambda_1$ and let $\rho$ be the spectral radius of its universal cover. Friedman (2003) proved  that every ``new'' eigenvalue of a random lift of $G$ is $O(\rho^{1/2}\lambda_1^{1/2} )$ with high probability, and conjectured a bound of $\rho+o(1)$, which would be tight by results of Lubotzky and Greenberg (1995).
Linial and Puder (2008) improved Friedman's bound to $O(\rho^{2/3}\lambda_1^{1/3})$.
For $d$-regular graphs, where $\lambda_1=d$ and $\rho=2\sqrt{d-1}$, this translates to a bound of $O(d^{2/3})$, compared to the conjectured $2\sqrt{d-1}$.

Here we analyze the spectrum of a random $n$-lift of a $d$-regular graph whose nontrivial eigenvalues are all at most $\lambda$ in absolute value.
We show that with high probability the absolute value of every nontrivial eigenvalue of the lift is $O((\lambda \vee \rho) \log \rho)$.
This result is tight up to a logarithmic factor, and for $\lambda \leq d^{2/3-\epsilon}$ it substantially improves the above upper bounds of Friedman and of Linial and Puder. In particular, it implies that a typical $n$-lift of a Ramanujan graph is nearly Ramanujan.
\end{abstract}

\maketitle

\vspace{-0.5cm}

\section{Introduction}


Over the last quarter of a century, \emph{expander} graphs have played a vital role in a remarkable variety of areas, ranging from combinatorics to discrete geometry to theoretical computer science, while exhibiting deep connections to algebra and number theory.
Notable applications of expanders, to name just a few, include the design of efficient communication networks, explicit error-correcting codes with efficient encoding and decoding schemes, derandomization of randomized algorithms, compressed sensing and the study of metric embeddings. See the expository article of Sarnak~\cite{Sarnak} on these intriguing objects, as well as the comprehensive survey of Hoory, Linial and Wigderson~\cite{HLW} demonstrating their many applications.

Informally, an expander is a graph where every small subset of the vertices has a relatively large edge boundary (see Section~\ref{sec:intro-exp} for a formal definition).
Most applications utilize $d$-regular sparse expanders ($d \geq 3$ fixed), where it is well-known that expansion is related to the ratio between $d$ and $\lambda$, the second largest eigenvalue in absolute value of the adjacency matrix. The smaller $\lambda$ is, the better the graph expansion becomes.
As a consequence of the Alon-Boppana bound~\cite{Nilli} (see also~\cite{Friedman2}) $\lambda \geq 2\sqrt{d-1}-o(1)$ where the $o(1)$-term tends to $0$ as the graph size tends to $\infty$. Graphs for which $\lambda \leq 2\sqrt{d-1}$ are in that respect optimal expanders and are called \emph{Ramanujan} graphs.

A proof that $d$-regular expanders  exist for any $d\geq 3$ was given by Pinsker~\cite{Pinsker} in the early 70's via a simple probabilistic argument. However, constructing good expanders \emph{explicitly} is far more challenging and particularly important in applications (see \cite{RVW} and the references therein), a task that was first achieved by Margulis~\cite{Margulis1}. Thereafter Ramanujan graphs were constructed explicitly in the seminal works of Lubotzky-Phillips-Sarnak~\cite{LPS} and Margulis~\cite{Margulis2}, relying on deep number theoretic facts. Till this date Ramanujan graphs remain mysterious: Not only are there very few constructions for such graphs, but for instance it is not even known whether they exist for any $d\geq 3$. A striking result of Friedman~\cite{Friedman08} shows that almost every $d$-regular graph on $n$ vertices is \emph{nearly} Ramanujan --- it has $\lambda=2\sqrt{d-1}+o(1)$ (the $o(1)$-term tends to $0$ as $n\to\infty$). What proportion of these graphs satisfy $\lambda \leq 2\sqrt{d-1}$ remains
an intriguing open problem.

The useful connection between expanders and the topological notion of \emph{covering maps} was
extensively studied by many authors over the last decade. Various properties of random covers of a given graph were thoroughly examined (see e.g.\ \cites{AL1,AL2,BL,LP}), motivated in part by the problem of generating good (large) expanders from a given one.

Given two simple graphs $G$ and $H$, a \emph{covering map} $\pi:V(H)\to V(G)$ is a homomorphism that for every $x\in V(H)$ induces a bijection between the edges incident to $x$ and those incident to $\pi(x)$. In the presence of such a covering map we say that $H$ is a \emph{lift} (or a cover) of $G$, or alternatively that $G$ is a \emph{quotient} of $H$. The \emph{fiber} of $y\in V(G)$ is the set $\pi^{-1}(y)$, and if $G$ is connected then all fibers are of the same cardinality, the \emph{covering number}.

 One well-known connection between covers and expansion is the fact that the universal cover of any $d$-regular graph is the infinite $d$-regular tree $\mathbb{T}_d$, whose spectral radius is $\rho=2\sqrt{d-1}$, the eigenvalue threshold in Ramanujan graphs. In fact, Greenberg and Lubotzky~\cite{Greenberg} (cf.\  \cite{Lubotzky}*{Chapter 4}) extended the Alon-Boppana bound to any family of general graphs in terms of the spectral radius of its universal cover (also see \cite{Friedman1}*{Theorem~4.1}).

It is easy to see that any lift of a $d$-regular base graph $G$ is itself $d$-regular and inherits all the original eigenvalues of $G$. One hopes that the lift would also inherit the expansion properties of its base graph, and in particular that almost every cover of a (small) Ramanujan graph will also be Ramanujan.

Since our focus here is on lifts of Ramanujan graphs (regular by definition) we restrict our attention to base graphs that are $d$-regular for $d\geq 3$.

\begin{figure}
\begin{center}
\subfigure{\includegraphics[width=3in]{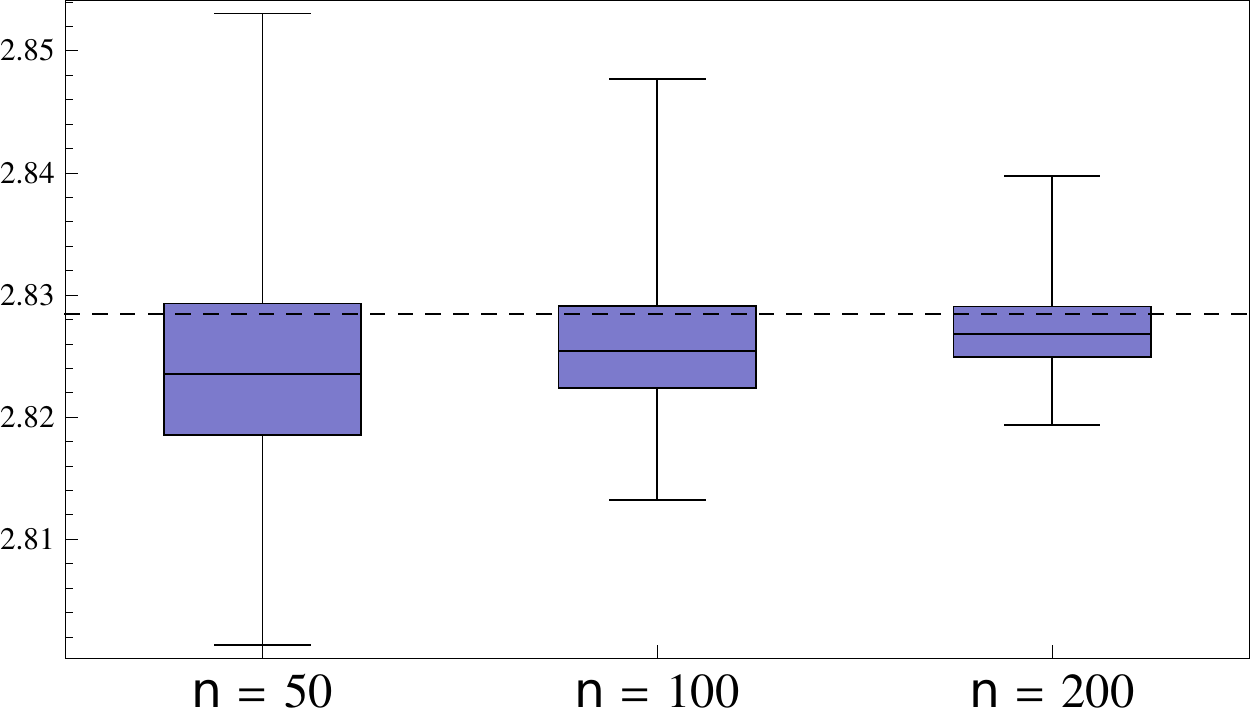}}
\hspace{0.1in}
\subfigure{\includegraphics[width=1.75in]{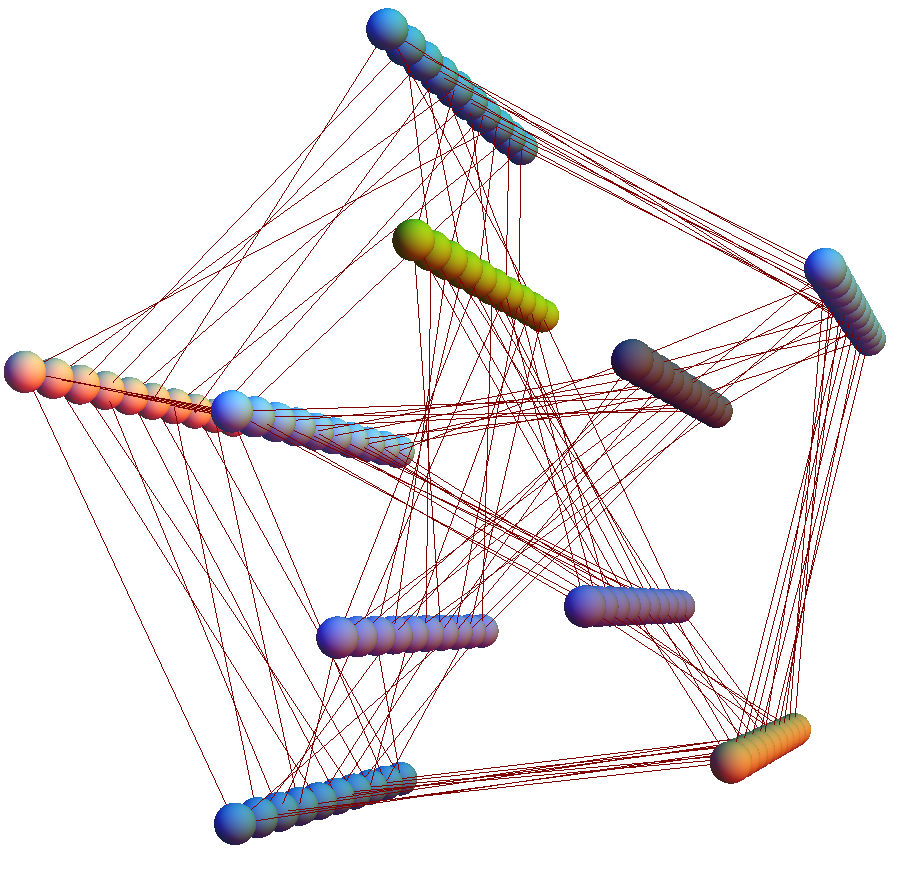}}
\end{center}
\vspace{-0.1in}
\caption{
Second eigenvalue (in absolute value) of a lifted Petersen graph, a 3-regular Ramanujan graph on 10 vertices,
simulated for covering number $n\in\{50,100,200\}$.
Dashed line marks the Ramanujan threshold $2\sqrt{2}$.
Boxes span values from the $\frac14$-quantile to the $\frac34$-quantile out of 1000 lifts.
\label{fig:liftpet}}
\end{figure}

A random uniform $n$-lift of a base graph $G$ (a uniform cover of $G$ with covering number $n$) has the following convenient description: It is the graph $H$ on the vertices $[n]\times V(G)$, where for each edge $u v$ in $G$ there is an independent uniform bijection $\pi$, and $H$ has all edges of the form $(i,u),(\pi(i),v)$.
The random lift of a complicated base-graph is thus a hybrid between the complex geometry of the quotient and the randomness due to the bijections.

In an important development in the study of the spectrum of random lifts Friedman~\cite{Friedman1} showed in 2003 that with high probability (w.h.p.) every ``new'' eigenvalue of an $n$-lift (one that is not inherited from the base graph) is at most $\sqrt{\rho \lambda_1}+o(1)$, where $\lambda_1$ is the largest eigenvalue of the base-graph and $\rho$ is the spectral-radius of its universal cover. When the base-graph $G$ is $d$-regular, $\rho=2\sqrt{d-1}$ and Friedman's result implies that in its random $n$-lift $H$ the largest absolute value of all nontrivial eigenvalues is w.h.p.
\begin{equation}
  \label{eq-Fri}
\lambda(H) \leq \lambda(G) \,\vee\, O(d^{3/4})\,,
\end{equation}
where $(a\vee b)$ denotes $\max\{a,b\}$. Conversely, $\lambda(H) \geq \lambda(G)$ and by Alon-Boppana it is also at least $2\sqrt{d-1} - o(1)$. This lower bound was conjectured by Friedman~\cite{Friedman1} to be tight (for general graphs he conjectured that all new eigenvalues are at most $\rho+o(1)$ as in the Greenberg-Lubotzky bound).

In a recent paper~\cite{LP}, Linial and Puder were able to significantly improve Friedman's bound and show that w.h.p.\ all the new eigenvalues of $H$ are at most $O(\rho^{2/3} \lambda_1^{1/3})$. Consequently, an $n$-lift $H$ of a $d$-regular $G$ w.h.p.\ satisfies
\begin{equation}
  \label{eq-LP}
  \lambda(H) \leq \lambda(G) \,\vee\, O(d^{2/3})\,.
\end{equation}
When $G$ is a $d$-regular expander with nontrivial eigenvalues of $O(\sqrt{d})$ as is the case for Ramanujan graphs, this translates to $c \sqrt{d} \leq \lambda(H) \leq O(d^{2/3})$.

Our main result in this work is the new near optimal upper bound of $ O((\lambda \vee \rho)\log\rho)$ when $G$ is $d$-regular with all nontrivial eigenvalues at most $\lambda$ in absolute value.
For $\lambda\leq d^{2/3-\epsilon}$ it substantially improves the known bounds \eqref{eq-Fri},\eqref{eq-LP}, and when $\lambda=O(\sqrt{d})$ as in Ramanujan graphs it is tight up to a logarithmic factor, giving $c\sqrt{d} \leq \lambda(H) \leq O(\sqrt{d}\log d)$.

\begin{maintheorem}\label{thm-1}
Let $G$ be a $d$-regular graph with all nontrivial eigenvalues at most $\lambda$ in absolute value and let $\rho=2\sqrt{d-1}$ be the spectral radius of its universal cover.
Let $H$ be a random $n$-lift of $G$. For some explicit absolute constant $C>0$,
every nontrivial eigenvalue of $H$ is at most $C(\lambda\vee\rho) \log \rho$ in absolute value
except with probability $O(n^{-100})$.
\end{maintheorem}

\begin{maincoro}\label{cor-2}
  Let $G$ be a $d$-regular Ramanujan graph vertices and let $H$ be a random $n$-lift of $G$. With probability $1-O(n^{-100})$ every nontrivial eigenvalue of $H$ is at most $C \sqrt{d}\log d$ in absolute value, where $C>0$ is an explicit absolute constant.
\end{maincoro}

Note that the above corollary implies that typical random $n$-lifts of Ramanujan graphs are nearly Ramanujan.
No attempt was made to optimize the explicit constant in Theorem~\ref{thm-1}. Finally, the statement of Theorem~\ref{thm-1} holds even when the size of the base-graph $m$ is allowed to grow with $n$ provided that $n$ is large enough in comparison (e.g., $n\geq m^{3/2}$).



\subsection{Related work}
The previous bounds on the spectra of random $n$-lifts of a fixed graph $G$ due to Friedman~\cite{Friedman1} and Linial and Puder~\cite{LP} were both obtained via Wigner's trace method.
The fact that the universal cover $\mathbb{T}$ of a connected graph $G$ is the infinite tree of non-backtracking walks from an arbitrarily chosen vertex makes the trace method particularly useful for relating the new eigenvalues of the lift with $\rho$, the spectral-radius of $\mathbb{T}$.

Even when the geometry of a graph is very well understood, bounding its nontrivial eigenvalues can be extremely challenging. For instance, a line of papers (cf.\ \cites{BS,FKS,Friedman1,Friedman2}) established various bounds for the second eigenvalue of certain random regular graphs, culminating in the optimal bound $2\sqrt{d-1}+o(1)$ for a uniformly chosen $d$-regular graph on $n$ vertices, proved by Friedman~\cite{Friedman08} using highly sophisticated arguments.

It turns out that this model is essentially the special case of an $n$-lift of a graph comprising a single vertex with self-loops: It is easy to see that for $d$ even, the random $d$-regular graph obtained by $d/2$ independent uniform permutations in $S_n$ is equivalent to an $n$-lift of the base-graph $G$ that has a single vertex with $d/2$ loops (this model is in fact contiguous to the uniform random $d$-regular graph for $d\geq 4$, cf.\ e.g.~\cite{Wormald}).
Unfortunately, when the base-graph features a complex and rich structure (e.g.\ the LPS-expanders, whose expansion properties hinge on a deep theorem of Selberg) it becomes significantly harder to control the spectrum of its lifts. Indeed, there are many examples of geometric properties that have been pinpointed precisely for the random regular graph yet remain unknown for arbitrary expanders (see \cite{LS} for a recent such example). Estimating the number of closed walks in lifts of arbitrary Ramanujan graphs thus appears to be a formidable task.

In this work, the bounds obtained for the spectra of lifts of arbitrary expanders rely on an approach introduced by Kahn and Szemer\'edi~\cite{FKS}, which is quite different from Wigner's trace method. This approach was originally used to control the spectrum of a random regular graph, and several new ideas are required to adapt it to the more complicated geometry of the lifts considered here.

Another related problem in the study of spectra of lifts, yet of a rather different nature, considers the $2$-lift of a base-graph (rather than $n$-lifts of a small fixed graph). Bilu and Linial~\cite{BL} showed that for any $d$-regular graph $G$ there \emph{exists} a $2$-lift with all new eigenvalues at most $O(\sqrt{d \log^3 d})$. This was shown by means of the Lov\'asz Local Lemma, combined with the crucial observation of~\cite{BL} whereby the new eigenvalues correspond precisely to the eigenvalues of a \emph{signing} of the adjacency matrix of $G$ (the matrix obtained by replacing a subset of its $1$ entries by $-1$). In the absence of such a characterization when the covering number $n$ is large, different tools are needed for the problem studied here, where we seek a bound that holds for almost every $n$-lift with $n$ sufficiently large.

\subsection{The distribution of the second eigenvalue}
As stated above, while a random $d$-regular graph $G$ has second eigenvalue $\lambda(G) \leq 2\sqrt{d-1} + o(1)$ w.h.p.\ (the $o(1)$-term tending to $0$ as $|V(G)|\to\infty$), the probability that $G$ is Ramanujan is unknown. See~\cites{HLW,Sarnak} for some experimental results suggesting that this probability is bounded away from $0$ and $1$.
As this is essentially the simplest special case of a random lift (the quotient being a single vertex with self-loops), it is natural to conjecture the following:
\begin{mainconj}\label{conj-3}
For any Ramanujan graph $G$ there exist some $0 < c < 1$ such that its random $n$-lift $H$ satisfies $\P(\mbox{$H$ is Ramanujan}) = c +o(1)$, where the $o(1)$-term tends to $0$ as $n\to\infty$.
\end{mainconj}
Note that the limiting constant in the above conjecture depends on the base-graph $G$,
as it is plausible that its structure may affect the probability of being Ramanujan. For instance, a random cover of a complete graph on $d+1$ vertices might behave quite differently compared to lifts of a sparse $d$-regular Ramanujan graph.
However, as we next elaborate, experimental results lead us to suspect that up to normalization this is not the case.

\begin{figure}
\centering
\makebox[\textwidth][c]{
\subfigure[Fiber size $n = 100$]{\includegraphics[width=3.1in]{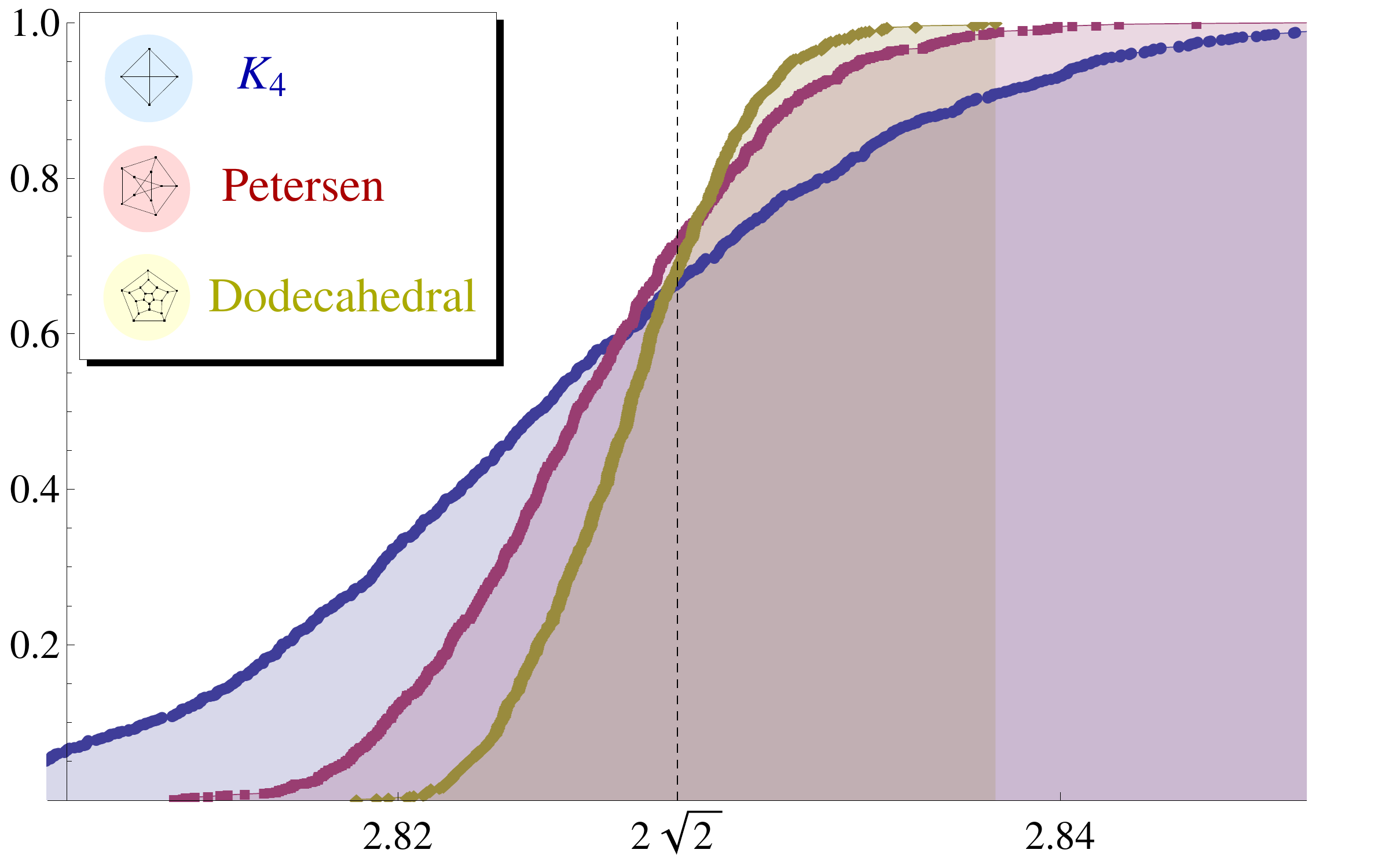}
    \label{fig:lifts:a}}
\hspace{-0.25in}
\subfigure[Total size $m n = 2000$]{\includegraphics[width=3.1in]{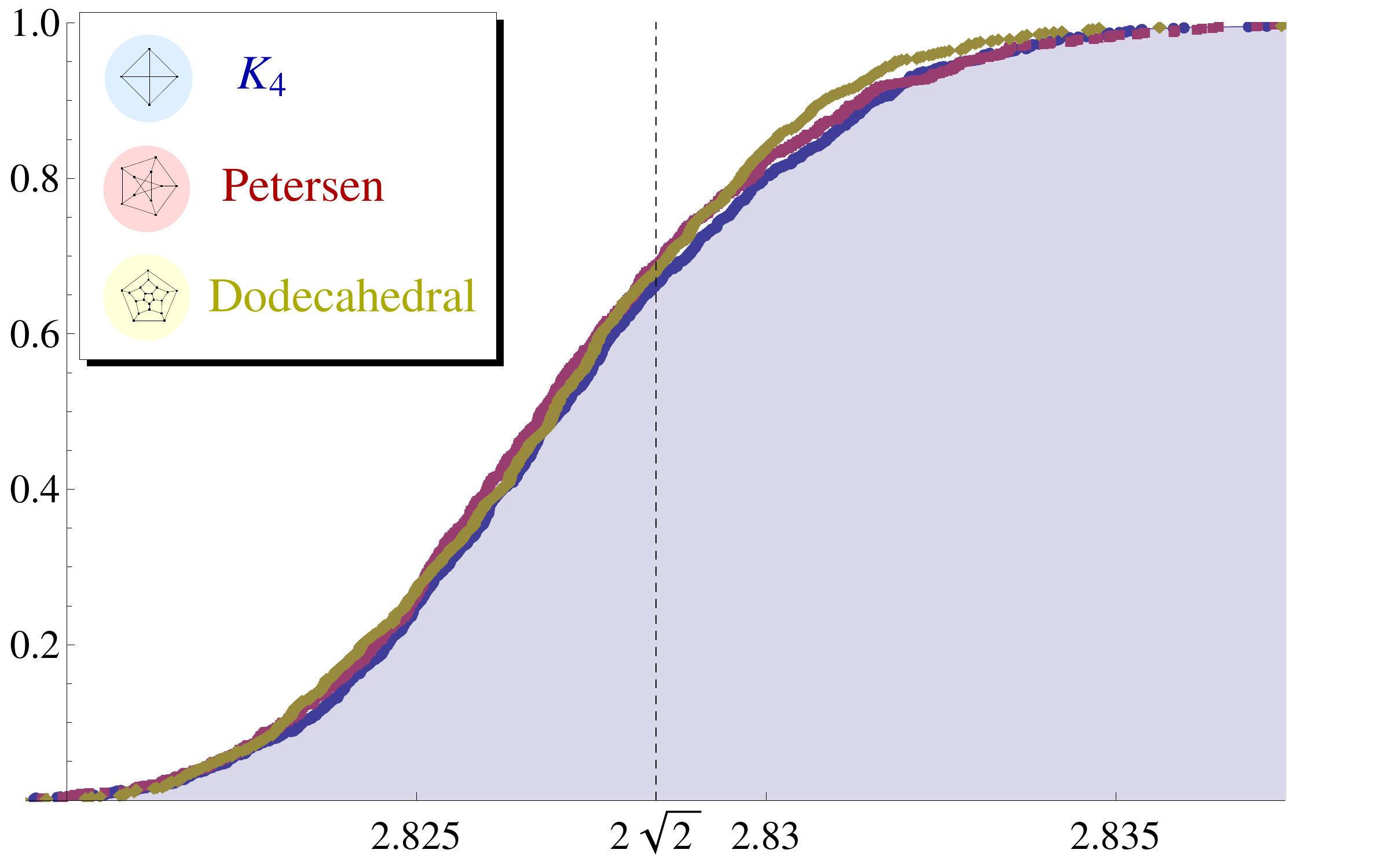}
    \label{fig:lifts:b}}
}
\vspace{-0.1in}
\caption{Empirical estimates for the c.d.f.\ of the second (in absolute value) eigenvalue of $n$-lifts of 3-regular Ramanujan graphs on $m$ vertices (1000 lifts were simulated per graph):
$K_4$ (complete graph on $m=4$ vertices), the
Petersen graph ($m=10$) and the Dodecahedral graph ($m=20$).
The c.d.f.'s coincide when aligning the total graph size. The probability for being strictly Ramanujan ($\lambda \leq 2\sqrt{2}$) is here roughly $2/3$.
\label{fig:lifts}}
\vspace{-0.15in}
\end{figure}

Figure~\ref{fig:lifts:a} shows the cumulative distribution function (c.d.f.) of $\lambda(H)$ where $H$ is the $100$-lift of 3 different $3$-regular Ramanujan base-graphs: $K_4$ (complete graph on $4$ vertices), the $10$-vertex Petersen graph and the $20$-vertex Dodecahedral graph. Each curve was evaluated from $1000$ random lifts.
In these simulations, the probability of a random lift being Ramanujan for each of these three base-graphs was bounded between $\frac35$ and $\frac45$.

Somewhat surprisingly, aligning the number of vertices of the graph cover $H$ to be the same (via $100$-lifts of the Dodecahedral graph, $200$-lifts of the Petersen graph and $500$-lifts of $K_4$, giving $2000$-vertex covers for each graph) resulted in the curves of the individual c.d.f.'s coinciding fairly accurately. This is demonstrated in Figure~\ref{fig:lifts:b}, in light of which we speculate that the following stronger version of the statement of Conjecture~\ref{conj-3} holds.

First, it seems plausible that for any integer $d \geq 3$ the limiting distribution of the second eigenvalue of the random cover is independent of the base-graph. Namely, there exists a distribution
$\mu_d$ on $[0,d]$ such that for any $d$-regular Ramanujan graph $G$ on $m$ vertices, the distribution of $\lambda(H)$ for its random $n$-lift $H$ converges to $\mu_d$ as $n\to\infty$. Second, the strong fit between the curves after aligning the total graph sizes suggests that even the rate of convergence to $\mu_d$ depends on $m n$ rather than on the geometry of the base-graph or even its relative size (in case $m$ is allowed to depend to $n$). Of-course, one clearly needs some level of ``burn-in'' for the covering number $n$ compared to $m$ since the cover $H$ starts as Ramanujan at $n=1$. For example, it may be that for any $n\geq m$ the total-variation distance between the distribution of $\lambda(H)$ and $\mu_d$ decays as a function of $mn$ alone, namely that $\| \P(\lambda(H) \in \cdot) - \mu_d\|_\mathrm{TV} \leq \alpha_d(m n)$ where $\alpha_d(k)$ depends only on $d,k$ and tends to $0$ as $k\to\infty$.

\section{Preliminaries and outline of the proof}

\subsection{Combinatorial vs.\ algebraic expanders}\label{sec:intro-exp}
The base-graph $G$ from Theorem~\ref{thm-1} corresponds to the algebraic definition of an expander known as an $(m,d,\lambda)$-graph.
An alternative closely-related criterion is the traditional definition of an expander graph in terms of its combinatorial edge or vertex expansion.
Let $G$ be a $d$-regular graph on $m$ vertices. The Cheeger constant of $G$ (also referred to as the edge isoperimetric constant) is defined as
\[ h(G) = \min_{\emptyset \neq S \subsetneqq V} \frac{|\partial S|}{|S|\;\wedge\; |V\setminus S|}\,,\]
where $(a\wedge b)$ denotes $\min\{a,b\}$ and $\partial S$ is the set of edges with exactly one endpoint in $S$. We say that $G$ is a $c$-edge-expander for some fixed $c > 0$ if it satisfies $h(G) > c$. Similarly, one defines
a $c$-vertex-expander by replacing $\partial S$ with the vertex boundary.

For $G$ as above the eigenvalues of the corresponding adjacency matrix are \[d=\lambda_1 \geq \lambda_2 \geq\ldots \geq \lambda_m \geq -d\] by Perron-Frobenius. We say that $G$ is an $(m,d,\lambda)$-graph if $|\lambda_i| \leq \lambda$ for all $i \neq 1$.
This notion was introduced by Alon in the 1980's, motivated by the fact that when $\lambda$ is much smaller than $d$ such graphs exhibit strong pseudo-random properties, resembling a random graph with edge density $d/m$. A notable example of this is captured by the \emph{Expander Mixing Lemma}: if $A,B$ are (not necessarily disjoint) subsets of vertices of an $(m,d,\lambda)$-graph then \begin{equation}
  \label{eq-exp-mixing}
\bigg| e(A,B) - \frac{d}{m} |A||B| \bigg| \leq \lambda\sqrt{ |A| |B|}\,,
\end{equation}
where $e(A,B) = \#\{(a,b):a \in A\,, b\in B\,, a b \in E(G)\}$ (\cite{AS}*{Chapter 9}).

Relating the above two notions of expansion is the following well-known discrete analogue of Cheeger's inequality bounding the first eigenvalue of a Riemannian manifold (Alon~\cite{Alon}, Alon-Milman~\cite{AM}, Dodziuk~\cite{Dodziuk}, Jerrum-Sinclair~\cite{JS}):
\[ \frac{d-\lambda}{2} \leq h(G) \leq \sqrt{2d(d-\lambda)}\]
See the survey \cite{HLW} for further information on expanders.

\subsection{Outline of the proof} We begin by describing the Kahn-Szemer\'edi \cite{FKS} approach for obtaining an $O(\sqrt{d})$ bound for random $d$-regular graphs.
Following Broder and Shamir \cite{BS}, the actual random graph model studied by \cite{FKS} is the $2d$-regular graph obtained from the union of $d$ permutations, contiguous to the lift of a single vertex with $d$ loops.

Let $H$ be the random graph in mention and let $A_H$ denote its adjacency matrix. By the Rayleigh quotient principle, the second (in absolute value) eigenvalue of the graph $H$ can be written as
\[ \lambda(H) = \max_{\substack{\|x\|= 1 \\ \left<x,\allone\right>=0}} \left| x^\tr A_H \, x\right| \,,\]
where $\allone=(1,\ldots,1)$ is the trivial eigenvector of $A_H$.
To bound $\lambda(H)$, the authors of \cite{FKS} analyzed the maximal possible value of $|x^\tr A_H \,y|$ separating the contribution of the pairs $x_i,y_j$ to the bilinear form into two cases:

\begin{enumerate}
\item \emph{Heavy pairs}: the contribution from those pairs
$x_i, y_j$ where $|x_i y_j|$ is suitably large. Here it is shown that w.h.p.\ the total contribution to $x^\tr A_H \,y$ by any pair of unit vectors $x,y \in \allone^\bot$ is at most $O(\sqrt{d})$.

\item \emph{Light pairs}: the remaining pairs $x_i,y_j$. Here it was shown that two fixed vectors $x,y$ are unlikely to contribute more than $O(\sqrt{d})$ to the bilinear form, and an $\epsilon$-net argument was used to extend this result to any unit vectors $x,y\in \allone^\bot$.
\end{enumerate}

Adapting this method to lifts of general graphs requires several additional ingredients.
Even in the simpler setting of \cite{FKS}, some of the arguments are only sketched and might prove difficult to complete in detail. More crucially, in our case we have little knowledge of the base-graph $G$, hence the study of both the ``heavy'' and ``light'' parts becomes significantly more involved.

First, our only input on $G$ is the magnitude of its second eigenvalue, which turns the analysis of the heavy part into a delicate optimization problem, requiring two levels of dyadic expansions of the potential contributions to the final bilinear form.

Second, the consideration of the light part relies on a non-trivial martingale argument which may be useful in other applications: In the absence of sufficient control over the expectation (due to the unknown contribution of the heavy part) we resort to an $L^2$ analysis of the increments in the corresponding Doob's martingale and apply a Bernstein-Kolmogorov type large deviation inequality due to Freedman.

\subsection{Notation}
Throughout the paper we use $G$ to denote the base-graph, a $d$-regular graph on $m$ vertices, and let $H$ denote its random $n$-lift. The asymptotic notation is used under the assumption that $n \rightarrow \infty$.

For the sake of clarity, when addressing a vertex in $V(H) = [n]\times V(G)$ we will typically denote it either by $ij$ or by $i'j'$ using indices $i,i'\in[n]$ and $j,j'\in[m]$.
Whenever $u,v$ are vertices in some graph whose identity is clear from the context, the abbreviation $u \sim v$ will denote that these two vertices are adjacent. For example, $ij \sim i'j'$ will usually stand for  $(ij,i'j')\in E(H)$, which in turn implies that $j\sim j'$ in $G$ by the definition of the lift.

Unless stated otherwise, all logarithms are using base $2$ and $\|\cdot\|$ denotes the $L^2$ norm in the appropriate Euclidean space.

\subsection{Organization}
The rest of this paper is organized as follows. Section~\ref{sec:heavy} deals with the contribution of the heavy pairs to the bilinear form $x^\tr A_H \,y$. Section~\ref{sec:light} deals with the contribution of the light pairs.
In the final section, Section~\ref{sec:mainproof}, we combine these results to conclude the proof of Theorem~\ref{thm-1}.

\section{Heavy pairs and large cuts}\label{sec:heavy}

Let $G$ be an $(m,d,\lambda)$-graph (that is, a $d$-regular graph where all nontrivial eigenvalues are at most $\lambda$ in absolute value) with adjacency matrix $A_G$, and $H$ be a random $n$-lift of $G$ with adjacency matrix $A_H$. As mentioned before, the largest nontrivial eigenvalue of $H$ in absolute value is precisely
\[ \max\left\{ \left| x^\tr A_H \, x\right| : x \in \R^{mn}\,,\, \left<x,\allone\right>=0\,,\,\|x\|= 1 \right\}\,,\]
where $\allone$ denotes the trivial eigenvector.

\subsection{Heavy pairs}
We first analyze the typical contribution to $x^\tr A_H \, y$ from pairs $x_{ij},y_{i'j'}$ with fairly large products.
More precisely, we say that a pair $x_{ij},y_{i'j'}$ is \emph{heavy} if $|x_{ij}y_{i'j'}| \geq \lambda/mn$, and otherwise it is \emph{light}.
For $x,y\in\R^{mn}$, define $R_h(x,y)$ to be the random variable
\[ R_h(x,y) = \sum_{ij\sim i'j'} x_{ij}y_{i'j'} \one_{\{|x_{ij}y_{i'j'}|\geq \lambda/mn\}} \,.\]
The next theorem estimates the contribution of the heavy pairs along the edges of $H$. The exponent of $m$ in the requirement $n \geq m^{3/2}$ was selected to simplify the exposition and can be replaced by $n \geq m^{1+\delta}$ for any $\delta > 0$.
\begin{theorem}\label{thm-heavy}
Let $G$ be an $(m,d,\lambda)$-graph with $\lambda \geq \sqrt{d}$ and let $H$ be a random $n$-lift of $G$
for $n\geq m^{3/2}$. Then with probability at least $1-O(n^{-100})$ every $x,y\in\R^{mn}$ with $\|x\|\leq 1$ and $\|y\|\leq1$ satisfy
$|R_h(x,y)| \leq 3500 \, \lambda \log d$ and moreover $|R_h(x,y) - \E [R_h(x,y)]| \leq 7000 \, \lambda \log d$.
\end{theorem}
The main ingredient in the proof of Theorem~\ref{thm-heavy} is the following lemma,
which provides an upper bound on the number of edges in a cut between subsets of vertices in $H$.
\begin{proposition}\label{prop-e(A,B)-lift-bound}
Let $H$ be a random $n$-lift of an $(m,d,\lambda)$-graph $G$ with $\lambda\geq \sqrt{d}$ and $n\geq m^{3/2}$.
 Then except with probability $O(n^{-100})$, every two subsets of vertices $A,B \subset V(H)$ with $|A||B| \leq (2mn/\lambda)^2$ satisfy
\begin{equation}
  \label{eq-e(A,B)-expander-bound}
e(A,B) \leq 802\,\lambda\sqrt{|A||B|} + 75(|A|+|B|)\log^2 d\,.
\end{equation}
\end{proposition}
We will next show how to derive Theorem~\ref{thm-heavy} from this lemma, whose proof is postponed to Subsection~\ref{sec-cuts-between-large-subsets}.
\begin{proof}[\emph{\textbf{Proof of Theorem~\ref{thm-heavy}}}]
Consider the following dyadic expansion of $x,y$:
\begin{align*}
\cD_\ell = &\bigg\{(i,j):  2^{\ell} \leq |x_{ij}| \sqrt{\frac{mn}{\lambda}}< 2^{\ell+1}\bigg\}&(\ell \in \Z)\,,\\
\cD'_{\ell'} = &\bigg\{(i',j'): 2^{\ell}\leq |y_{i'j'}| \sqrt{\frac{mn}{\lambda}}< 2^{\ell'+1}\bigg\}&(\ell' \in \Z)\,,
\end{align*}
and identify any element $(i,j)$ in $\cD_\ell$ or $\cD_\ell'$ with the vertex $ij$ in $H$. These definitions,
together with the assumption on $\|x\|$ and $\|y\|$, implies that
\begin{align}\label{eq-H-ell-l2-bound}
  \sum_\ell 4^\ell |\cD_\ell| \frac{\lambda}{mn} \leq \sum_{i,j} x_{ij}^2 \leq 1\,,\quad
  \sum_{\ell'} 4^{\ell'} |\cD'_{\ell'}| \frac{\lambda}{mn} \leq \sum_{i',j'} y_{i'j'}^2 \leq 1\,.
\end{align}
Furthermore, if $ij\in\cD_\ell$ and $i'j'\in \cD'_{\ell'}$ then a necessary condition for
$|x_{ij} y_{i'j'}| \geq \lambda/mn$ is that $\ell+\ell'+2 > 0$, and so
\begin{equation}
  \label{eq-heavy-pairs-dyadic}
\sum_{ij\sim i'j'} \left|x_{ij}y_{i'j'}\right| \one_{\{|x_{ij}y_{i'j'}|\geq \lambda/mn\}}\leq
4 \sum_{\ell+\ell' > -2} 2^{\ell+\ell'} e(\cD_\ell,\cD'_{\ell'}) \frac{\lambda}{mn}\,,
\end{equation}
where we $e(\cD_\ell,\cD'_{\ell'})$ is the number of edges between the two subsets of vertices in $H$ corresponding to $\cD_\ell$ and $\cD_\ell'$.

To prove that $|R_h(x,y)| = O(\lambda\log d)$, we set
\[ D = \log\Big(\frac{d}{\lambda\log d}\Big)-3 \,,\]
and analyze the sum in \eqref{eq-heavy-pairs-dyadic} according to whether or not $|\ell-\ell'| > D$.

First, consider $D \leq 0$. As $H$ is $d$-regular, trivially $e(\cD_\ell,\cD'_{\ell'}) \leq d |\cD_\ell|$, yielding
\begin{align*}
  \sum_{\substack{\ell+\ell'> -2\\ \ell \geq \ell'}} 2^{\ell+\ell'}
  e(\cD_{\ell},\cD'_{\ell'})\frac{\lambda}{mn} &\leq \sum_{\ell}4^\ell \cdot d|\cD_\ell|\cdot \frac{\lambda}{mn}
  \sum_{\ell \geq \ell'} 2^{-(\ell-\ell')} \leq 2d\,,
\end{align*}
where the second inequality used \eqref{eq-H-ell-l2-bound}. Similarly,
\begin{align*}
  \sum_{\substack{\ell+\ell'> -2\\ \ell' \geq \ell}} 2^{\ell+\ell'}
  e(\cD_{\ell},\cD'_{\ell'})\frac{\lambda}{mn} &\leq 2d\,,
\end{align*}
and since $D \leq 0$ occurs if and only if $d \leq 8\lambda\log d$, altogether in this case
\begin{align}\label{eq-D<0}
  \sum_{\ell+\ell'> -2} 2^{\ell+\ell'}
 e(\cD_{\ell},\cD'_{\ell'})\frac{\lambda}{mn} &\leq 4d \leq 32\lambda\log d\,.\end{align}

We now focus on $D > 0$. Consider the case where $\ell-\ell' > D > 0$. Repeating the above argument, we now get
\begin{align*}
  \sum_{\substack{\ell+\ell'> -2\\ \ell-\ell'> D}} 2^{\ell+\ell'}
  e(\cD_{\ell},\cD'_{\ell'})\frac{\lambda}{mn} &\leq \sum_{\ell}4^\ell \cdot d|\cD_\ell|\cdot \frac{\lambda}{mn}
  \sum_{\ell-\ell'>D} 2^{-(\ell-\ell')}\\
  &\leq d 2^{-D} \leq 8\lambda\log d\,,
\end{align*}
By symmetry, the same argument holds for the case $\ell'-\ell > D$, and we infer that
\begin{align}\label{eq-ell-ell'-geq-D}
  \sum_{\substack{\ell+\ell'> -2\\ |\ell'-\ell|> D}} 2^{\ell+\ell'}
  e(\cD_{\ell},\cD'_{\ell'})\frac{\lambda}{mn} &\leq 16\lambda\log d\,.\end{align}
It remains to treat $|\ell-\ell'| \leq D$ with $D > 0$. This will be achieved with the help of Proposition~\ref{prop-e(A,B)-lift-bound}, which estimates the size of the cut between two subsets $A,B$ in case $|A||B| \leq (2mn/\lambda)^2$.
Indeed, for $\ell+\ell' \geq -1$ we have 
\begin{align*}
\frac14 |\cD_\ell||\cD'_{\ell'}| \frac{\lambda^2}{(mn)^2} &\leq 4^{\ell+\ell'} |\cD_\ell||\cD'_{\ell'}| \frac{\lambda^2}{(mn)^2}
\leq \sum_{ij\in \cD_\ell}\sum_{i'j'\in\cD'_{\ell'}} x_{ij}^2 y_{i'j'}^2 \\
&\leq \sum_{ij} x_{ij}^2 \sum_{i'j'} y_{i'j'}^2 \leq 1\,,
\end{align*}
and therefore $|\cD_\ell||\cD'_{\ell'}| \leq (2mn/\lambda)^2$.
Thus, Proposition~\ref{prop-e(A,B)-lift-bound} gives that with probability $1-O(n^{-100})$,
\begin{align}
  \sum_{\substack{\ell+\ell'> -2\\ |\ell-\ell'| \leq D}} &2^{\ell+\ell'}
  e(\cD_{\ell},\cD'_{\ell'})\frac{\lambda}{mn} \leq  802 \sum_{\substack{\ell+\ell'>-2\\ |\ell-\ell'| \leq D}} 2^{\ell+\ell'}\frac{\lambda}{mn}
  \cdot \lambda\sqrt{|\cD_{\ell}||\cD'_{\ell'}|} \nonumber\\
 & +     75 \sum_{\substack{\ell+\ell'> -2\\ |\ell-\ell'| \leq D}} 2^{\ell+\ell'}\frac{\lambda}{mn}
 \cdot (|\cD_{\ell}|+|\cD'_{\ell'}|)\log^2 d\,.\label{eq-heavy-bound-small-D}
\end{align}
For the first expression in the right-hand-side of \eqref{eq-heavy-bound-small-D}, note that
there are at most $2D+1 \leq 2\log (d/\lambda) \leq \log d $ integers $k$ such that $|k| \leq D$ (here we used the fact that $\lambda\geq\sqrt{d}$).  For each such value, we can combine \eqref{eq-H-ell-l2-bound}
with Cauchy-Schwartz to get that
\begin{align*}
 \sum_{\substack{\ell+\ell'> -2\\ \ell-\ell' = k}} 2^{\ell+\ell'}\frac{\lambda}{mn}
  \cdot \lambda\sqrt{|\cD_{\ell}||\cD'_{\ell'}|} \leq \lambda \sqrt{\sum_{\ell} 4^{\ell}|\cD_\ell|\frac{\lambda}{mn}}
  \sqrt{\sum_{\ell'} 4^{\ell'}|\cD'_{\ell'}|\frac{\lambda}{mn}} \leq \lambda\,,
\end{align*}
and summing over $k$ it follows that
\begin{align*}
  \sum_{\substack{\ell+\ell'> -2\\ |\ell-\ell'| \leq D}} 2^{\ell+\ell'}\frac{\lambda}{mn} \cdot\lambda\sqrt{|\cD_{\ell}||\cD'_{\ell'}|} \leq  \lambda \log d \,.
\end{align*}
For the second expression in \eqref{eq-heavy-bound-small-D}, again recall that $\lambda\geq \sqrt{d}$, and so
\begin{align*}
\sum_{\substack{\ell+\ell'> -2\\ |\ell-\ell'| \leq D}} &2^{\ell+\ell'}\frac{\lambda}{mn}
 |\cD_{\ell}|\log^2 d \leq \sum_{\ell} 4^{\ell} \frac{\lambda}{mn}
 |\cD_{\ell}|\log^2 d \sum_{|k| \leq D}2^{-k} \\
 &< 2^{D+1}\log^2 d = \frac{d}{4\lambda\log d }\log^2 d \leq \frac14 \lambda\log d\,,
 \end{align*}
 and the same applies to the analogous quantity for $|\cD'_{\ell'}|$.

Altogether, these two estimates for \eqref{eq-heavy-bound-small-D} sum up to
\begin{align*}
  \sum_{\substack{\ell+\ell'> -2\\ |\ell-\ell'| \leq D}} &2^{\ell+\ell'}
e(\cD_{\ell},\cD'_{\ell'})\frac{\lambda}{mn} \leq (802 + 2\cdot \tfrac{75}4)\lambda \log d < 840\,\lambda\log d\,,\end{align*}
and combining this with \eqref{eq-D<0} and \eqref{eq-ell-ell'-geq-D} gives that
\begin{align*}
\sum_{\ell+\ell'> -2} 2^{\ell+\ell'}
  e(\cD_{\ell},\cD'_{\ell'})\frac{\lambda}{mn} &\leq 856\, \lambda\log d\,.
\end{align*}
Recalling \eqref{eq-heavy-pairs-dyadic}, we deduce that
\begin{align*}
|R_h(x,y)| \leq  4 \sum_{\ell+\ell'> -2} 2^{\ell+\ell'}
  e(\cD_{\ell},\cD'_{\ell'})\frac{\lambda}{mn} &\leq 3424\, \lambda\log d\,.
\end{align*}
To obtain the statement on $|R_h(x,y)-\E[R_h(x,y)]|$ first note that since $H$ is $d$-regular
$|R_h(x,y)| \leq\lambda_1(H) = d$ with probability $1$ for any two unit vectors $x,y$.
We have already established that, except with probability $O(n^{-100})$, every pair of vectors $x,y$ with norm at most $1$ satisfies $|R_h(x,y)|\leq 3424\, \lambda\log d$. Hence,
\[ \E |R_h(x,y)| \leq 3424 \, \lambda \log d + d O(n^{-100}) < 3425 \, \lambda \log d\,,\]
where the last inequality holds for any sufficiently large $n$. Reapplying the result on $R_h(x,y)$
(along with the triangle inequality) now completes the proof of the theorem (with room to spare).
\end{proof}

\subsection{Proof of Proposition~\ref{prop-e(A,B)-lift-bound}}\label{sec-cuts-between-large-subsets}
Write $|A| = \alpha n$ and $|B| = \beta n$ where $0 < \alpha,\beta \leq m$. Our assumption on $|A||B|$ then translates into
\begin{equation}
  \label{eq-assumption-alpha-beta}
   \alpha \beta \leq (2m/\lambda)^2\,.
\end{equation}
We aim to show that, except with probability $1-O(n^{-100})$, for any such $A,B$ we have
$e(A,B) =O\big(\lambda\sqrt{|A||B|} + (|A|+|B|)\log^2 d\big)$, or in terms of $\alpha,\beta$, that
\begin{equation*}
  e(A,B)/n = O\left(\lambda\sqrt{\alpha\beta} + (\alpha+\beta) \log^2 d\right)\,.
\end{equation*}
Define the following partition of the fibers according to a dyadic expansion of their proportion that is included in $A$.
\begin{align*}
S_i &= \left\{v \in V(G) : 2^{-i-1} < \frac{\left|A \cap ([n] \times \{v\})\right|}{n} \leq 2^{-i}  \right\}~(i=0,1,\ldots,\log n)\,,\\
A_i &= A \cap \bigcup_{v \in S_i} ([n] \times \{v\}) \,,\quad s_i = |S_i|\,,\quad \alpha_i = s_i 2^{-i}\,.
\end{align*}
Notice that by these definitions, $s_i$ is the number of fibers with about $2^{-i} n$ vertices from $A$, and so $|A_i|\approx s_i 2^{-i} n = \alpha_i n$. In other words, there are about $\alpha_i n$ vertices of $A$ in fibers of type $S_i$, and more precisely,
\begin{align}\label{eq-alphai-alpha}\tfrac12 \alpha_i n < |A_i| \leq \alpha_i n\,,\quad \tfrac12 \sum_i \alpha_i < \alpha \leq \sum_i \alpha_i\,.\end{align}
Similarly, we perform an analogous dyadic expansion for $B$:
\begin{align*}
T_j &= \left\{v \in V(G) : 2^{-j-1} < \frac{\left|B \cap ([n]\times\{v\})\right|}{n} \leq 2^{-j}  \right\}~(j=0,1,\ldots,\log n)\,,\\
B_j &= B \cap \bigcup_{v \in T_j} ([n]\times\{v\}) \,,\quad t_j = |T_j|\,,\quad \beta_j = t_j 2^{-j}\,,
\end{align*}
and again have that
\begin{align}\label{eq-betai-beta}\tfrac12\beta_j n < |B_j| \leq \beta_j n\,,\quad \tfrac12\sum_j \beta_j < \beta\leq \sum_j \beta_j\,.\end{align}
Clearly,
\[ e(A,B) = \sum_{i,j} e(A_i, B_j)\,,\]
and our bound on $e(A,B)$ will follow from an analysis of the number of edges between the various types of $A_i$'s and $B_j$'s.

First, consider the case $i = j = 0$. Here we have $\frac12 s_0 n < |A_0| \leq |A|$ and $\frac12 t_0 n < |B_0| \leq |B|$. Since there are $n$ edges in $H$ between any pair of fibers that correspond to adjacent vertices in $G$, the Expander Mixing Lemma (see~\eqref{eq-exp-mixing}) applied to the base-graph $G$ gives that
\begin{align*}
 e(A_0,B_0) &\leq n e(S_0, T_0) \leq \frac{dn}m s_0 t_0 + n \lambda \sqrt{s_0 t_0}
 < \frac{4d}{mn} |A||B| + 2\lambda \sqrt{|A||B|} \\ &= \frac{4d}{m} \alpha\beta n + 2\lambda \sqrt{\alpha\beta}n\,.
\end{align*}
Recalling that $\sqrt{\alpha\beta} \leq 2m/\lambda$ (see \eqref{eq-assumption-alpha-beta}) it follows that
\begin{align}\label{eq-cut-A0-B0}
 e(A_0,B_0) &\leq 2\sqrt{\alpha\beta}n \Big(\frac{2d}{m} \cdot \frac{2m}{\lambda} + \lambda\Big)
 = 2\sqrt{\alpha\beta}n \Big(\frac{4d}{\lambda}+\lambda\Big) \leq 10\lambda \sqrt{\alpha\beta}n\,,
\end{align}
where the last inequality used the fact that $\lambda \geq \sqrt{d}$ and so $\lambda \geq d/\lambda$.

Next, consider $e(A_i,B_j)$ in case $|i-j| > 2 \log d$. There is a total of $s_i$ fibers in $A_i$, thus by definition of the $n$-lift of a $d$-regular graph there are at most $d s_i$ fibers, where $B_j$ may have vertices that contribute to $e(A_i,B_j)$.
Since $B_j$ has at most $2^{-j} n$ vertices in each fiber, and each vertex has $d$ neighbors in $H$, we deduce that
\begin{align*}
&\sum_{j - i > 2 \log d} e(A_i,B_j) \leq \sum_i \sum_{j > i + 2\log d}  d s_i \cdot 2^{-j} n \cdot d\\
&\quad=\sum_i   s_i 2^{-i} n  \cdot d^2\sum_{j-i > 2\log d} 2^{-(j-i)} \leq \sum_i s_i2^{-i}n = \sum_i \alpha_i n
\leq 2\alpha n\,,
\end{align*}
where the last inequality followed from \eqref{eq-alphai-alpha}. Similarly, we have
\begin{align*}
\sum_{i - j > 2 \log d} e(A_i,B_j) \leq \sum_j \beta_j n \leq 2\beta n \,,
\end{align*}
and conclude that
\begin{align}\label{eq-cut-diff-geq-logd}
\sum_{|i - j| > 2 \log d} e(A_i,B_j) \leq 2(\alpha+\beta)n\,.
\end{align}
It remains to treat the case $|i-j| \leq 2\log d$ for all $(i,j)\neq (0,0)$, where the required bound will only hold w.h.p.

Consider a prescribed set of $k$ pairs of vertices $(i_l j_l, i'_l j'_l)$ ($l\in[k]$) in $H$. We wish to bound the probability that $\{i_l j_l \sim i'_l j'_l \mbox{ for all $l$}\}$ by $(3/n)^k$. By the independence of the different pairs of fibers in the lift, it clearly suffices to show this when all the $i_l j_l$'s are on one fiber and all the $i'_l j'_l$'s are on another, i.e., for some $j\neq j'$ and all $l$ we have $j_l = j$ and $j'_l=j'$.
When $k \leq \frac23 n$ then it is straightforward that this probability is indeed at most $(3/n)^k$. To see this, expose the pairings of $i_1j ,i_2 j,\ldots,i_k j$ one by one, and note that for $l \leq k \leq \frac23 n$, the probability to match $i_l j$ to $i'_l j'$, given that so far we succeeded in matching all the $l-1$ previous pairs, is $1/(n-l+1) \leq 3/n$.

Further note that, when considering potential edges between $A_i$ and $B_j$, the case $k \geq \frac23 n$ can only arise when $i=j=0$, otherwise no two fibers have more than $n/2$ points of $A_i$ and $B_j$ respectively. Since we excluded the case $i=j=0$, the above estimate holds for any of our sets $A_i,B_j$.

Write
\[W_{ij} = e(A_i, B_j) / n\,,\] and recall that $A_i$ and $B_j$ are contained in the fibers corresponding to $S_i$ and $T_j$ resp., and have at most $2^{-i}n$ and $2^{-j}n$ vertices on each of these respective fibers. Suppose first that the identity of the fibers $S_i$ and $T_j$ are given (we will account for these later). In this case, the number of configurations of the vertices of $A_i$ on the fibers $S_i$ can be upper bounded by $2^{s_0 n}$ if $i=0$ and by $\binom{n}{2^{-i} n}^{s_i} \leq 2^{(i+2)s_i 2^{-i}n}$ if $i \neq 0$, here using the well-known inequality $\binom{a}b \leq (\mathrm{e}a/b)^b$. Similarly, the number of configuration of $B_j$ on the fibers $T_j$ is at most $2^{(j+2)t_j 2^{-j}n}$.
For each such configuration of the vertices of $A_i,B_j$ there are at most $2^{-i-j}n^2 e_G(S_i,T_j)$ pairs which may potentially be connected in $H$. From the above estimate on the probability of $k$ pairs being adjacent in $H$, it now follows that for any $w_{ij} > 0$ and choice of $S_i$ and $T_j$,
\begin{align*}
  \P(W_{ij} = w_{ij}) &\leq 2^{(i+2) s_i 2^{-i} n}2^{(j+2) t_j 2^{-j} n} \binom{2^{-i-j} n^2 e(S_i,T_j)}{w_{ij} n} \left(\frac3n\right)^{w_{ij} n}\\
  &\leq 2^{\left((i+2) \alpha_i   + (j+2) \beta_j \right)n} \left(9 \cdot 2^{-i-j} e(S_i,T_j)/w_{ij}\right)^{w_{ij}n}\,.
\end{align*}
Defining \[ z_{ij} = \frac{2^{i+j} \; w_{ij}}{9 e(S_i, T_j)}\,,\]
we then get
\begin{align}
  \P(W_{ij} = w_{ij}) &\leq \left(2^{(i+2) \alpha_i   + (j+2) \beta_j} 2^{- w_{ij} \log z_{ij}}\right)^n  \nonumber\\
   &= 2^{ \left[(i+2) \alpha_i   + (j+2) \beta_j - 9 e(S_i,T_j)2^{-i-j} z_{ij} \log z_{ij} \right] n } \,.\label{eq-prob-Wij}
\end{align}
We will next establish a threshold for $z_{ij}$ such that the above probability would be at most $\exp(-n^{3/4-o(1)})$ and then translate this bound to the cut $e(A_i,B_j)$ via the corresponding $w_{ij}$'s.

Consider the equation $x\log x = b$ and note that for
$b > 0$ it has a unique solution $x> 1$ monotone increasing in $b$. Let $z_{ij}^\star$ be the solution to
\begin{equation}  \label{eq-z-star}
  z^\star_{ij} \log z^\star_{ij} = \frac{2^{i+j}}{9e(S_i,T_j)}\left[(i+2)\alpha_i + (j+2)\beta_j + n^{-1/4}\right]\,,
\end{equation}
and define its counterpart (similar to the relation between $w_{ij}$ and $z_{ij}$)
\begin{equation}
  \label{eq-w-star}
w_{ij}^\star = \frac{9e(S_i,T_j)}{2^{i+j}} \left( z_{ij}^\star \vee 2\right)\,.
\end{equation}
Combining these definitions with the probability bound \eqref{eq-prob-Wij}, while noting that this bound in that equation is monotone decreasing in $z_{ij}$ (and hence in $w_{ij}$) in the range $z_{ij} \geq 1$, we deduce that
for any $k \geq w_{ij}^\star$,
\[ \P(W_{ij} = k) \leq 2^{-n^{3/4}}\,.\]
Since $W_{ij}=e(A_i,B_j)/n$ with $e(A_i,B_j) \leq e(H)=dmn/2$ we can sum $k$ over all possible values that $W_{ij}$ can accept and infer that
\[ \P(W_{ij} \geq w_{ij}^\star) \leq dm 2^{-n^{3/4}} \,.\]
Next recall that the above estimate was for $A_i,B_j$ with a given choice of the fibers $S_i,T_j$. Summing the above probability over all possible choices for such fibers (using a trivial bound of $2^m$ options for each of the sets) and then further summing over at most $\log^2 n$ pairs of $i,j$ we deduce that
\[ \P\left(\cup_{i,j} \{ e(A_i,B_j) \geq w_{ij}^\star n\}\right) \leq (d m  \log^2 n )2^{2m - n^{3/4}} < n^{-100}\,,\]
with the last inequality valid for any sufficiently large $n$ since $m \leq n^{2/3}$.
Collecting~\eqref{eq-cut-A0-B0} and~\eqref{eq-cut-diff-geq-logd} this yields that,
except with probability $n^{-100}$, any two sets $A,B$ with $\alpha\beta\leq (2m/\lambda)^2$ (as per \eqref{eq-assumption-alpha-beta}) satisfy
\begin{equation}
   \label{eq-e(A,B)-via-wij}
   e(A,B)/n \leq \sum 10\lambda\sqrt{\alpha\beta} + 2(\alpha+\beta) + \sum_{\substack{i+j > 0\\ |i-j|\leq 2\log d}} w_{ij}^\star\,.
 \end{equation}
It thus suffices to bound $\sum_{i,j} w_{ij}^\star $ in order to complete the proof.
\begin{lemma}\label{lem-wij-star}
Let $z_{ij}^\star$ be the solution to \eqref{eq-z-star} and let $w_{ij}^\star$ be its counterpart as given in \eqref{eq-w-star}. Then for any $d \geq 320$ we have
\[ \sum_{\substack{i+j > 0\\ |i-j|\leq 2\log d}} w^\star_{ij} \leq 792\lambda\sqrt{\alpha\beta} + 74 (\alpha+\beta)\log^2 d + n^{-1/3}\,.\]
\end{lemma}
\begin{proof}
We first consider pairs $i,j$ such that $z_{ij}^\star \geq 2$.

By the definition of $z_{ij}^\star$ as the solution of \eqref{eq-z-star}, the right-hand-side of that equation, which we denote by $b_{ij}$, necessarily then satisfies
\begin{equation}
  \label{eq-assumption-z-2}
b_{ij} = \frac{2^{i+j}}{9e(S_i,T_j)}\left[(i+2)\alpha_i + (j+2)\beta_j + n^{-1/4}\right] \geq 2\,.
\end{equation}
Furthermore, since for any $b > 1$ the solution of $z \log z = b$ satisfies $z < 2\frac{b}{\log b}$ (this is easy to verify using the monotonicity of $z\log z$) we can infer an upper bound on $z_{ij}^\star$ in the form of
\begin{align*}
z_{ij}^\star \leq 2  \frac{2^{i+j}}{9e(S_i,T_j)\log b_{ij}}\left[(i+2)\alpha_i + (j+2)\beta_j + n^{-1/4}\right]\,,
\end{align*}
and as a consequence
\begin{align*}
 w_{ij}^\star\one_{\{z^\star_{ij}\geq 2\}} &= \frac{9e(S_i,T_j)}{2^{i+j}} z_{ij}^\star \leq \frac{2}{\log b_{ij}} \left[(i+2)\alpha_i + (j+2)\beta_j +
 n^{-1/4}\right]\,.
\end{align*}
Immediately by \eqref{eq-assumption-z-2} the last denominator is at least $1$ and so
\[ \sum_{|i-j|\leq 2\log d} \frac{n^{-1/4}}{\log b_{ij}}
\leq  \sum_{i,j} n^{-1/4} \leq n^{-1/4} \log^2 n = o(n^{-1/3})\,,\]
where $0 \leq i,j \leq \log n$ by definition.
Next, note that for the same reason
\begin{align*}
  \sum_{|i-j| \leq 2\log d}
\frac{\alpha_i + \beta_j }{\log b_{ij}}
&\leq  (1+4\log d) \bigg(\sum_{i} \alpha_i + \sum_j \beta_j\bigg) \\
&\leq (2+8\log d)(\alpha+\beta)\,,
\end{align*}
using the fact that $\sum_i \alpha_i \leq 2\alpha $ and $\sum_j \beta_j \leq 2\beta$ as given in \eqref{eq-alphai-alpha},\eqref{eq-betai-beta}. The combination of the last three equations
implies that for large enough $n$,
\begin{align}
\sum_{\substack{i+j>0\\ |i-j| \leq 2\log d}}  w_{ij}^\star \one_{\{z^\star_{ij}\geq 2\}}&\leq
2 \sum_{|i-j|\leq 2\log d} \frac{i\alpha_i + j\beta_j }{\log b_{ij}} \nonumber\\
&+ (32\log d+8)(\alpha+\beta)+ n^{-1/3}\,.\label{eq-w-star-sum}
\end{align}
Before we further analyze the expressions $(i\alpha_i + j\beta_j)/\log b_{ij}$ we wish to narrow down the range of pairs $(i,j)$. First, we can quickly move to $i,j \geq 20$. Indeed, if for instance $i < 20$ then using \eqref{eq-assumption-z-2}
\begin{align*} \sum_{\substack{i < 20 \\ |i-j|\leq 2\log d}} &\frac{i\alpha_i + j\beta_j }{\log b_{ij}} \leq \sum_{\substack{i\leq 19\\ j \leq 19 + 2\log d}} i\alpha_i + j\beta_j \\
&\leq (40\log d + 380)\bigg(\sum_i\alpha_i + \sum_j\beta_j\bigg) \leq \left(80\log d + 760\right)(\alpha+\beta)\,.\end{align*}
An analogous calculation holds for $j<20$, yielding that
\begin{align*} \sum_{\substack{i < 20 \\ |i-j|\leq 2\log d}} &\frac{i\alpha_i + j\beta_j }{\log b_{ij}}
+ \sum_{\substack{j < 20 \\ |i-j|\leq 2\log d}} &\frac{i\alpha_i + j\beta_j }{\log b_{ij}} \leq \left(160\log d + 1520\right)(\alpha+\beta)\,.\end{align*}
The case $20 \leq i,j \leq 4\log d$ is treated similarly:
\begin{align*} \sum_{20\leq i,j \leq 4\log d} \frac{i\alpha_i +j\beta_j}{\log b_{ij}}  &\leq 4\log d (4\log d - 19) \bigg(\sum_i \alpha_i + \sum_j \beta_j\bigg) \\
&\leq (32\log^2 d - 152\log d)(\alpha+\beta) \,.\end{align*}
Plugging the last two equations in \eqref{eq-w-star-sum} and defining
\begin{equation}\label{eq-ij-regime}\Gamma = \left\{ (i,j) : \begin{array}{c} i,j \geq 20\,,\\
|i-j| \leq 2\log d\,,\\
i \geq 4\log d ~\mbox{ or }~ j \geq 4 \log d\,.
 \end{array}
\right\}\end{equation}
it follows that
\begin{align}
\sum_{\substack{i+j>0\\ |i-j| \leq 2\log d}}  &w_{ij}^\star \one_{\{z^\star_{ij}\geq 2\}}\leq
2 \sum_{(i,j)\in\Gamma} \frac{i\alpha_i + j\beta_j }{\log b_{ij}} \nonumber\\
&+ (32\log^2 d + 40\log d + 1528)(\alpha+\beta)+ n^{-1/3}\,.\label{eq-w-star-sum2}
\end{align}
Recalling the definition~\eqref{eq-assumption-z-2} of $b_{ij}$, it now remains to bound $\sum_{(i,j)\in\Gamma}\xi_{ij}$ with $\xi_{ij}$ given by
\[\xi_{ij} = \frac{i\alpha_i + j\beta_j }{\log \left[ \frac{2^{i+j}[(i+2)\alpha_i + (j+2)\beta_j + n^{-1/4}]}{9e(S_i,T_j)}\right]} \leq \frac{i\alpha_i + j\beta_j }{\log\left[ \frac{2^{i+j}(i\alpha_i + j\beta_j)}{9e(S_i,T_j)}\right]}\,.\]
Note that the last inequality (where we reduced the argument of the $\log(\cdot)$)
is only legitimate provided that
\begin{equation}
  \label{eq-log-positive}
  2^{i+j}\frac{i\alpha_i + j\beta_j}{9 e(S_i,T_j)} > 1\,.
\end{equation}
In what follows we will show that this is indeed the case and then proceed
to bound $\sum \xi_{ij}$. This will be achieved by splitting the analysis into two cases, according to the structure of $e(S_i,T_j)$ in the base graph $G$.
Recall that we have 
$e(S_i,T_j) \leq (d/m) s_i t_j + \lambda\sqrt{s_i t_j}$ as $G$ is an $(m,d,\lambda)$-graph.
\begin{list}{\labelitemi}{\leftmargin=1em}
\item \textbf{Case (i):} $e(S_i,T_j) \leq 2(d/m) s_i t_j$

Since $\alpha_i = s_i 2^{-i}$ and $\beta_j = t_j 2^{-j}$, in this case we have
\[ \frac{2^{i+j}}{e(S_i,T_j)} \geq \frac{2^{i+j}}{2(d/m)s_i t_j} = \frac{m}{2d \alpha_i \beta_j}\,.\] With the regime of $(i,j)$ as in \eqref{eq-ij-regime} in mind, suppose first that $i \geq 4\log d$.
It follows that
\begin{align*} \sum_{\substack{i\geq 4\log d\\ j\geq 20\\ |i-j|\leq2\log d}} \xi_{ij} &\leq
\sum_{\substack{i\geq 4\log d\\ j\geq 20\\ |i-j|\leq2\log d}} \frac{i\alpha_i+ j\beta j}{\log\left[ \frac{m(i\alpha_i + j\beta_j)}{18 d \alpha_i \beta_j}\right]} \leq \sum_{\substack{i\geq 4\log d\\ j\geq 20\\ |i-j|\leq2\log d}} \frac{i\alpha_i + j\beta_j}{\log\frac{m j}{18 d \alpha_i}}
\,.
\end{align*}
As $\alpha_i 2^i = s_i \leq m$ we have $\alpha_i \leq m 2^{-i}$ and
plugging in the fact that $j \geq 20$,
\begin{align*}
 \sum_{\substack{i \geq 4\log d \\ j\geq20 \\ |i-j|\leq2\log d}} &\frac{i\alpha_i+j\beta_j}{\log\frac{mj}{18 d\alpha_i}} \leq
 \sum_{\substack{i \geq 4\log d \\ |i-j|\leq2\log d}} \frac{i\alpha_i+j\beta_j}{\log(2^i/d)} \leq
 \sum_{\substack{i \geq 4\log d \\ |i-j|\leq2\log d}} \frac{i\alpha_i+j\beta_j}{\frac34i} \\
&\leq  \frac43 \sum_{|i-j|\leq2\log d}\alpha_i + 2 \sum_{|i-j|\leq2\log d}\beta_j \leq (1+4\log d)(3 \alpha + 4\beta)\,.
\end{align*}
where we used the fact that $j \leq i + 2\log d \leq \frac32 i$ for the above $i,j$.
Note that we have just verified Eq.~\eqref{eq-log-positive} by showing that its left-hand-side
is at least $2^i/d \geq d^3$.
Similarly, if $j \geq 4\log d$ then
\begin{align*}
 \sum_{\substack{j \geq 4\log d \\ i\geq20 \\ |i-j|\leq2\log d}} \xi_{ij} \leq
 (1+4\log d)(4 \alpha + 3\beta)\,.
\end{align*}
Altogether, in this case we have
\begin{align}
\label{eq-heavy-mixing-lemma-case-i}
\sum_{(i,j)\in\Gamma} \xi_{ij} \leq (28\log d + 7)(\alpha+\beta)\,.
\end{align}

\item \textbf{Case (ii):} $e(S_i,T_j) \leq 2\lambda \sqrt{s_i t_j}$

Rewriting the assumption on $e(S_i,T_j)$ in terms of $\alpha_i$ and $\beta_j$, we have
\[ \frac{2^{i+j}}{e(S_i,T_j)} \geq \frac{2^{i+j}}{2\lambda\sqrt{s_i t_j}} = \frac{2^{\frac{i+j}2}}{2\lambda\sqrt{\alpha_i \beta_j}}\,,\] which gives that
\[ \xi_{ij} \leq \frac{i \alpha_i + j \beta_j}{\log\left(\frac{i \alpha_i + j \beta_j}{18\lambda\sqrt{\alpha_i\beta_j}}2^{\frac{i+j}2} \right)} \leq
 \frac{i \alpha_i + j \beta_j}{\log\left(\frac{\sqrt{ij}}{18\lambda} 2^{\frac{i+j}2}\right)}
\leq  \frac{i \alpha_i + j \beta_j}{\log\left(2^{\frac{i+j}2}/\lambda\right)}
\,,\]
where the second inequality was derived from the fact that $x+y \geq 2\sqrt{xy}$ for any $x,y\geq 0$, and the last one by the fact that $i,j\geq 20$.

Notice that if $2^{\frac{i+j}4} \leq \lambda$ then $i,j \leq 4\log d$ and thus $(i,j)\notin\Gamma$.
We therefore have $2^{\frac{i+j}4} > \lambda$ and so
\[ \log\left(2^{\frac{i+j}2}/\lambda\right) > \log\left(2^{\frac{i+j}4}\right) = (i+j)/4\,.\]
This verifies \eqref{eq-log-positive} and further implies that $\xi_{ij} \leq 4 (\alpha_i + \beta_j)$.
Altogether,
\begin{align} \sum_{(i,j)\in\Gamma} \xi_{ij} &\leq 4\sum_{|i-j|\leq 2 \log d} (\alpha_i + \beta_j)
\leq 4(1+4\log d) \Big(\sum_{i}\alpha_i+\sum_j \beta_j\Big) \nonumber\\
&\leq (32\log d + 8 )(\alpha+\beta)\,.\label{eq-heavy-mixing-lemma-case-ii}
\end{align}
%
\end{list}
Combining \eqref{eq-w-star-sum2} with the two cases \eqref{eq-heavy-mixing-lemma-case-i},\eqref{eq-heavy-mixing-lemma-case-ii} for the $\xi_{ij}$'s proves that
\begin{align}
  \sum_{\substack{i+j > 0\\ |i-j|\leq 2\log d}} w^\star_{ij}\one_{\{z_{ij}^\star \geq 2\}} &\leq
(32\log^2 d + 160\log d + 1558)(\alpha+\beta) + n^{-1/3}\nonumber\\
&\leq 74 (\alpha+\beta) \log^2 d + n^{-1/3}\,,    \label{eq-sum-wij-under-z-2}
\end{align}
where in the last inequality we plugged in the fact that $d \geq 320$.

It remains to consider the case $z_{ij}^\star < 2$ where by definition
\[w_{ij}^\star \one_{\{z_{ij}^\star < 2\}} = 18 \frac{e(S_i,T_j)}{2^{i+j}}\,.\]
Since the $(m,d,\lambda)$-graph $G$ satisfies $e(S_i,T_j) \leq (d/m) s_i t_j + \lambda\sqrt{s_i t_j}$,
we have the following two cases:
\begin{list}{\labelitemi}{\leftmargin=1em}
\item \textbf{Case (i):} $e(S_i,T_j) \leq 2(d/m) s_i t_j$

The above bound on $w^\star_{ij}$ then translates into
\[ w^\star_{ij} \one_{\{z_{ij}^\star < 2\}} \leq 36\frac{d}{m}\frac{s_i t_j}{2^{i+j}} = 36\frac{d}{m}\alpha_i\beta_j\,,\]
and summing over all such $i,j$ while recalling that $\sqrt{\alpha\beta} \leq 2m/\lambda$ we get
\begin{align}\sum_{i,j} w^\star_{ij} \one_{\{z_{ij}^\star < 2\}} &
\leq 36\frac{d}{m}\sum_i\alpha_i \sum_j \beta_j \leq 144\frac{d}{m}\alpha\beta  \nonumber\\
&\leq 144\frac{d}{m}\sqrt{\alpha\beta}\frac{2m}{\lambda}  \leq 288\lambda\sqrt{\alpha\beta}\,,\label{eq-xi-case-i}
\end{align}
where we used the inequalities $\sum_i \alpha_i \leq 2\alpha$, $\sum_j \beta_j \leq 2\beta$ and $\lambda\geq\sqrt{d}$.

\item \textbf{Case (ii):} $e(S_i,T_j) \leq 2\lambda \sqrt{s_i t_j}$

Here we have
\[ w^\star_{ij} \one_{\{z_{ij}^\star < 2\}}\leq 36\lambda\frac{\sqrt{s_i t_j}}{2^{i+j}} = \frac{36\lambda}{2^{(i+j)/2}}\sqrt{\alpha_i\beta_j}\,,\]
and so
\begin{align*}\sum_{i,j} &w^\star_{ij} \one_{\{z_{ij}^\star < 2\}} 
\leq 36\lambda\sum_{i,j}\sqrt{\frac{\alpha_i}{2^i}} \sqrt{\frac{\beta_j}{2^j}}\\
&= 36\lambda\sum_{k \geq 0} \sum_{i}\sqrt{\frac{\alpha_i}{2^i}} \sqrt{\frac{\beta_{i+k}}{2^{i+k}}} +
36\lambda\sum_{k > 0} \sum_{j}\sqrt{\frac{\beta_j}{2^j}} \sqrt{\frac{\alpha_{j+k}}{2^{j+k}}}
\,.
\end{align*}
By Cauchy-Schwartz,
\begin{align*}\sum_{k \geq 0}& \sum_{i}\sqrt{\frac{\alpha_i}{2^i}} \sqrt{\frac{\beta_{i+k}}{2^{i+k}}} \leq
\sum_{k \geq 0} \sqrt{\Big(\sum_{i}\frac{\alpha_i}{2^i}\Big)\Big(\sum_i \frac{\beta_{i+k}}{2^{i+k}}\Big)} \\
&\leq \sum_{k\geq 0}\sqrt{2^{-k}\sum_i\alpha_i \sum_j\beta_j}
\leq 2\sqrt{\alpha\beta} \sum_{k\geq 0}2^{-k/2} \leq 7 \sqrt{\alpha\beta}\,,
\end{align*}
and similarly,
\begin{align*}\sum_{k > 0} \sum_{j}\sqrt{\frac{\beta_j}{2^j}} \sqrt{\frac{\alpha_{j+k}}{2^{j+k}}} \leq
 7 \sqrt{\alpha\beta}\,.
\end{align*}
We deduce that in this case
\begin{equation} \sum_{i,j} w^\star_{ij} \one_{\{z_{ij}^\star < 2\}} 
\leq 504 \lambda\sqrt{\alpha\beta}\,.\label{eq-xi-case-ii}\end{equation}
\end{list}
Adding the bounds obtained for the two cases \eqref{eq-xi-case-i},\eqref{eq-xi-case-ii} gives
\begin{align*}
  \sum_{i,j} w^\star_{ij} \one_{\{z_{ij}^\star < 2\}}\leq 792 \lambda\sqrt{\alpha\beta} \,.
\end{align*}
The proof of Lemma~\ref{lem-wij-star} is now concluded by combining the above bound with \eqref{eq-sum-wij-under-z-2}.
\end{proof}
Notice that for proving Proposition~\ref{prop-e(A,B)-lift-bound} we can assume that $d \geq 320$, since otherwise
$d < 5 \log^2 d $ and the statement immediately follows from the trivial bound \[e(A,B)\leq d(|A| \wedge |B|) \leq \frac{d}2(|A|+|B|) < \frac52(|A|+|B|)\log^2 d\,.\]

For $d \geq 320$ we can apply Lemma~\ref{lem-wij-star} (recalling the discussion preceding this lemma) combined with
\eqref{eq-e(A,B)-via-wij} and obtain that, except with probability $n^{-100}$,
every two subsets $A,B$ with $\alpha\beta\leq (2m/\lambda)^2$ satisfy
\begin{equation*}
  e(A,B)/n \leq 802\lambda\sqrt{\alpha\beta} + 74 (\alpha+\beta)\log^2 d + 2(\alpha + \beta) + n^{-1/3}\,.
\end{equation*}
When the subsets $A,B$ satisfy in addition
\[
  |A|+|B| \geq n^{2/3}\]
then $\alpha+\beta \geq n^{-1/3}$ and the above bound (for $d\geq 320$) translates to
\begin{equation}
    \label{eq-large-sets-bound}
 e(A,B)/n \leq 802\lambda\sqrt{\alpha\beta} + 75 (\alpha+\beta)\log^2 d \,.
\end{equation}
Altogether, we have established the statement of Proposition~\ref{prop-e(A,B)-lift-bound} under the additional
assumption~$|A|+|B|\geq n^{2/3}$ for the subsets $A,B$ in mention.

The separate case of $|A|+|B| < n^{2/3}$ is much simpler to handle, and is treated by the next claim
using a standard first moment argument.
\begin{claim}\label{cl-small-A,B}
Let $G$ be an arbitrary graph on $m$ vertices and let $H$ be a random $n$-lift of $G$ with $n\geq m$.
Then with probability $1-O(n^{-100})$, every two subsets $A,B\subset V(H)$ of size $|A|+|B|\leq n^{2/3}$ have
$e(A,B)\leq 50\big(|A|+|B|\big)$.
\end{claim}
\begin{proof}
Suppose that $A,B$ are two subsets that satisfy $|A|+|B|\leq n^{2/3}$ and $e(A,B) \geq 50(|A|+|B|)$,
and consider their union $R = A\cup B$.
Clearly, \[|R| \leq |A|+|B| \leq n^{2/3}\] whereas the number of edges in the induced subgraph on $R$ satisfies \[e(R) \geq e(A,B)/2 \geq 25 |R|\,.\]
As argued below Eq.~\eqref{eq-cut-diff-geq-logd}, if $(i_1 j_1,i'_1 j'_1),\ldots,(i_k j_k,i'_k j'_k)$
are $k$ arbitrary distinct pairs of vertices in $H$ of which no $\frac23 n$ vertices share the same fiber,
\[\P\left( i_l j_l \sim i'_l j'_l\mbox{ for all }l\in[k]\right) \leq (3/n)^k\]
(that argument applies to any base-graph $G$ by definition of the $n$-lift). Consider $k$ distinct pairs of vertices in $R$ that may potentially be adjacent in $H$. Clearly, for large enough $n$ these do not contain
any $\frac23 n$ points on the same fiber since $|R| = o(n)$, hence the probability that they are all adjacent is at most $(3/n)^k$.

It now follows that the probability there exists a subset $R\subset V(H)$ of size $|R|=r \leq n^{2/3}$ with $e(R) \geq 25 r$ is at most
\[ \binom{mn}r \binom{\binom{r}2}{25r} \left(\frac3n\right)^{25r} \leq
\left(\frac{\mathrm{e}mn}r \left(\frac{3\mathrm{e}r}{50 n}\right)^{25} \right)^r \leq \left(c m (r/n)^{24}\right)^r
\]
where $c>0$ is an absolute constant. As $r \leq n^{2/3}$ and $m \leq n$, the base of the exponent in the last expression is at most $O(n^{-7})$ whereas $r \geq 50$ necessarily to allow $e(R) \geq 25r$. Summing this error probability over the $n^{2/3}$ possible values of $r$ completes the proof.
\end{proof}
As the bound given in the above claim is clearly smaller than the bound
\eqref{eq-large-sets-bound} for $d\geq 320$, together they imply that \eqref{eq-large-sets-bound}
holds with probability $1-O(n^{-100})$ for any two sets $A,B$ with $|A||B|\leq(2mn/\lambda)^2$.
This concludes the proof of Proposition~\ref{prop-e(A,B)-lift-bound}.
\qed

\section{Light pairs and epsilon-nets}\label{sec:light}
We now move on to estimating the expected total of all light pairs along the edges of $H$.
Recall that a pair $x_{ij},y_{i'j'}$ is \emph{light} if $|x_{ij}y_{i'j'}| < \lambda/mn$.
To bound the bilinear form $x^\tr A_H\, y$ with respect to the light pairs, we will approximate each such vector using an $\epsilon$-net, where
\[\epsilon = \frac{1}{d\sqrt{mn}}\,.\]
More precisely, we consider the $mn$-dimensional lattice $\cL = (\epsilon \Z)^{mn}$, and show that the required statement on the bilinear form holds for any two vectors $x,y$ with norm at most $1$ in this lattice.
\begin{theorem}\label{thm-light}
Let $G$ be an $(m,d,\lambda)$-graph for $d\geq 3$ and $\lambda \geq \sqrt{d}$ and let $H$ be a random $n$-lift of $G$.
For $x,y\in\R^{mn}$, let $R_l(x,y)$ be the random variable
\[ R_l(x,y) = \sum_{ij \sim i'j'} x_{ij}y_{i'j'} \one_{\{|x_{ij}y_{i'j'}|< \lambda/mn\}} \,.\]
Let $\cL$ denote the $mn$-dimensional lattice $\left(\frac{1}{d\sqrt{mn}}\Z\right)^{mn}$.
Then except with probability $O(\exp(-mn))$, every $x,y\in\cL$ with $\|x\|\leq 1$ and $\|y\|\leq1$ satisfy
$|R_l(x,y) - \E [R_l(x,y)]| \leq 250\, \lambda \log d$.
\end{theorem}
\begin{proof}
In order to establish the above concentration result, we must first estimate the variance of $R_l(x,y)$.
\begin{lemma}\label{lem-variance-Rl}
Let $G$ be an $(m,d,\lambda)$-graph, and let $x,y\in\R^{mn}$ be two fixed vectors satisfying $\|x\|\leq 1$ and $\|y\|\leq 1$.
Then \[\sum_{i,i'}\sum_{j\sim j'} x_{ij}^2 y_{i'j'}^2 \one_{\{|x_{ij}y_{i'j'}|<\lambda/mn\}} \leq 50 \frac{\lambda^2 \log d}m\,.\]
\end{lemma}
\begin{proof}
As in the treatment of the heavy pairs, we consider the following dyadic expansion of $x$ and $y$:
\begin{align*}
\cD_\ell = &\bigg\{(i,j):  2^{-\ell} \leq |x_{ij}| \sqrt{\frac{mn}{\lambda}}< 2^{-\ell+1}\bigg\}&(\ell \in \Z)\,,\\
\cD'_{\ell'} = &\bigg\{(i',j'): 2^{-\ell'}\leq |y_{i'j'}| \sqrt{\frac{mn}{\lambda}}< 2^{-\ell'+1}\bigg\}&(\ell' \in \Z)\,,
\end{align*}
and the assumptions $\|x\|\leq 1$ and $\|y\|\leq 1$ translate into
\begin{align}\label{eq-L-ell-l2-bound}
  \sum_\ell 4^{-\ell} |\cD_\ell| \frac{\lambda}{mn} \leq 1\,,\quad
  \sum_{\ell'} 4^{-\ell'} |\cD'_{\ell'}| \frac{\lambda}{mn} \leq 1\,.
\end{align}
Further note that, if $ij\in\cD_\ell$ and $i'j'\in \cD'_{\ell'}$ then a necessary condition for
$|x_{ij} y_{i'j'}| < \lambda/mn$ is that $\ell+\ell'> 0$.

Consider the graph $G'$ where every two fibers that are connected in $G$ have a complete bipartite graph between them in $G'$. That is, $ij\sim i'j'$ in $G'$ if $jj'\in E(G)$.
It follows that
\begin{align}
\sum_{j\sim j'} \sum_{i,i'} x_{ij}^2 y_{i'j'}^2 &\one_{\{|x_{ij}y_{i'j'}|< \lambda/mn\}}\leq
 16\sum_{\ell+\ell'> 0} 4^{-(\ell+\ell')} e_{G'}(\cD_\ell,\cD'_{\ell'}) \Big(\frac{\lambda}{mn}\Big)^2
\,, \label{eq-light-pairs-dyadic}
\end{align}
and we aim to bound the sum in the right-hand-side by at most $3\frac{\lambda^2 \log d}m$.

The adjacency matrix of $G'$ is therefore precisely $A_G \otimes J_n$, where $J_n$ is the all-ones matrix of order $n$ and $\otimes$ denotes tensor product, and so by the definition of $G$ (and the properties of tensor products) it follows that $G'$ is an $(mn,dn,\lambda n)$-graph.
As such, for any two subsets $A,B$ of its vertices,
\begin{equation}
  \label{eq-exp-mixing-tensor-graph}
  e_{G'}(A,B) \leq \frac{d}m |A||B| + \lambda n \sqrt{|A||B|}\,.
\end{equation}

We now separate the sum in \eqref{eq-light-pairs-dyadic} into two cases, comparing $|\ell-\ell'|$ to
\[ D = \log \Big(\frac{d}{\lambda \log d}\Big) + 2 \leq \frac12\log d - 1\,.\]
To justify the last inequality, note first that we may assume that $d \geq 256$
otherwise the statement of the lemma holds trivially. Indeed,
since there are $dm/2$ edges in $G$, summing over $\frac12 dmn^2$ pairs, each of which contributes at most $(\lambda/mn)^2$, gives
at most $\frac{\lambda^2 d}{2m}$. For $d \leq 256$ we have $d \leq 32\log d$ and so this is clearly at most $16 \frac{\lambda^2\log d}m$ and we are done.
Assume therefore that $d\geq 256$, in which case $D+1 = \log\big(\frac{8d}{\lambda \log d}\big) \leq \log(d/\lambda)$ and the above inequality follows from the fact that $\lambda \geq \sqrt{d}$.

In case $D \leq 0$ we have $d \leq \frac14\lambda\log d$. Here, applying the trivial bound $e_{G'}(\cD_\ell,\cD'_{\ell'}) \leq d n |\cD_\ell|$ gives
\begin{align*}
  \sum_{\substack{\ell+\ell'> 0\\ \ell' \geq \ell}} 2^{-(\ell+\ell')}
  e_{G'}(\cD_{\ell},\cD'_{\ell'})\Big(\frac{\lambda}{mn}\Big)^2 &\leq \frac{\lambda}{mn} \sum_{\ell}4^{-\ell} \cdot d n|\cD_\ell|\cdot \frac{\lambda}{mn}
  \sum_{\ell' \geq \ell} 2^{-(\ell'-\ell)} \\
  &\leq \frac{2\lambda d}{m} \leq \frac{\lambda^2 \log d}{2m}\,,
\end{align*}
where the inequality between the two lines used \eqref{eq-L-ell-l2-bound}. Performing the same calculation for the sum over $\ell \geq \ell'$ gives the same bound. Altogether these two bounds sum up to $\frac{\lambda^2\log d}m$ and we get that for $D \leq 0$
\begin{align*}
\sum_{\ell+\ell'>0} &4^{-(\ell+\ell')}
 e_{G'}(\cD_{\ell},\cD'_{\ell'})\Big(\frac{\lambda}{mn}\Big)^2 \\
 &\leq  \frac12 \sum_{\ell+\ell'>0} 2^{-(\ell+\ell')}
 e_{G'}(\cD_{\ell},\cD'_{\ell'})\Big(\frac{\lambda}{mn}\Big)^2 \leq \frac{\lambda^2\log d}{2m}
 \end{align*}
 (the factor of $\frac12$ in the first inequality due to the fact that $\ell+\ell' > 0$).
 Thus, Eq.~\eqref{eq-light-pairs-dyadic} translates this bound to $8\frac{\lambda^2 \log d}m$ and confirms the statement of the lemma for the case $D \leq 0$.
It remains to handle $D > 0$.
\begin{list}{\labelitemi}{\leftmargin=1em}
\item \textbf{Case (i):} $|\ell-\ell'| \geq D$

In this case, we use the trivial bound $e_{G'}(\cD_\ell,\cD'_{\ell'}) \leq d n |\cD_\ell|$, giving
\begin{align*}
&  \sum_{k \geq D}  \sum_{\substack{\ell+\ell'> 0 \\ \ell'-\ell=k}} 4^{-(\ell+\ell')} e_{G'}(\cD_\ell,\cD'_{\ell'}) \Big(\frac{\lambda}{mn}\Big)^2 \leq
  \sum_{k \geq D} \sum_{\substack{\ell+\ell' > 0 \\ \ell'-\ell = k}} 2^{-4\ell-2k} dn |\cD_\ell|  \Big(\frac{\lambda}{mn}\Big)^2  \\
  &=  \sum_\ell 4^{-\ell} \frac{\lambda}{mn} |\cD_\ell| \sum_{\substack{k \geq D\\ \ell+(k+\ell)> 0}} 2^{-(2\ell+k)} (2^{-k} d) \frac{\lambda}m \leq \sum_\ell 4^{-\ell} \frac{\lambda}{mn} |\cD_\ell| \frac{\lambda^2\log d}{4m} \,,
\end{align*}
where we used the facts that $2^{k} \geq 4d/(\lambda\log d)$ for $k \geq D$, and that $\sum 2^{-(2\ell+k)} \leq 1$. By \eqref{eq-L-ell-l2-bound}, it now follows that
\begin{align*}
  \sum_{k \geq D}  \sum_{\substack{\ell+\ell'> 0 \\ \ell'-\ell=k}} 4^{-(\ell+\ell')} e_{G'}(\cD_\ell,\cD'_{\ell'}) \Big(\frac{\lambda}{mn}\Big)^2 \leq \frac{\lambda^2\log d}{4m}\,,
\end{align*}
and adding the symmetric case where we sum over $\ell-\ell'=k$, we get
\begin{align*}
\sum_{\substack{\ell+\ell'> 0 \\ |\ell'-\ell| \geq D}} &4^{-(\ell+\ell')}  e_{G'}(\cD_\ell,\cD'_{\ell'}) \Big(\frac{\lambda}{mn}\Big)^2 \leq \frac{\lambda^2\log d}{2m}\,.
\end{align*}

\item \textbf{Case (ii):} $|\ell-\ell'| < D$

Here we have two bounds according to the two expressions in the upper bound \eqref{eq-exp-mixing-tensor-graph}.
That is, we break $e_{G'}(\cD_\ell,\cD_{\ell'})$ into the sum of the two expressions corresponding to $\frac{d}m |\cD_\ell||\cD_{\ell'}|$
and to $\lambda n \sqrt{|\cD_\ell||\cD_{\ell'}|}$ and bound each of them separately.

First, by~\eqref{eq-L-ell-l2-bound} the sum corresponding to $\frac{d}m |\cD_\ell||\cD_{\ell'}|$ contributes
\begin{align*}
   \sum_{0\leq k < D}  &\sum_{\substack{\ell+\ell'> 0 \\ \ell'-\ell=k}} 4^{-(\ell+\ell')} \frac{d}m |\cD_\ell||\cD'_{\ell'}| \Big(\frac{\lambda}{mn}\Big)^2 \\
   &\leq
  \lceil D \rceil \frac{d}m \sum_{\ell} 4^{-\ell} \frac{\lambda}{mn} |\cD_\ell|
   \sum_{\ell'} 4^{-\ell'} \frac{\lambda}{mn} |\cD'_{\ell'}|\leq \lceil D\rceil \frac{d}m \leq \frac{\lambda^2 \log d}{2m}\,,
\end{align*}
where the last inequality is due to the facts $\lambda^2 \geq d$ and $\lceil D\rceil \leq \frac12 \log d$.
Second, the sum corresponding to $\lambda n\sqrt{|\cD_\ell||\cD_{\ell'}|}$ contributes
\begin{align*}
   \sum_{0\leq k < D}  &\sum_{\substack{\ell+\ell'> 0 \\ \ell'-\ell=k}} 4^{-(\ell+\ell')} \lambda n \sqrt{|\cD_\ell||\cD'_{\ell'}|} \Big(\frac{\lambda}{mn}\Big)^2 \\
   &\leq
   \lceil D\rceil \frac{\lambda^2}m \sum_{\substack{\ell+\ell'> 0}}
   \sqrt{2^{-2\ell}\frac{\lambda}{mn} |\cD_\ell|}
   \sqrt{2^{-2\ell'}\frac{\lambda}{mn} |\cD'_{\ell'}|} 2^{-(\ell+\ell')}    \\
   &\leq \frac{\lceil D\rceil}2 \frac{\lambda^2}m
   \sqrt{\sum_\ell 2^{-2\ell}\frac{\lambda}{mn} |\cD_\ell|}
   \sqrt{\sum_{\ell'} 2^{-2\ell'}\frac{\lambda}{mn} |\cD'_{\ell'}|} \leq \frac{\lceil D\rceil}2 \frac{\lambda^2}m \leq \frac{\lambda^2 \log d}{4m} \,,
\end{align*}
where the first inequality of the last line followed from Cauchy-Schwartz and the last one used the fact $\lceil D\rceil \leq \frac12\log d$. The last two inequalities now give a combined bound of $\frac{3\lambda^2\log d}{4m}$.
As the same holds for the sum over $\ell-\ell'=k$, altogether we have
\[\sum_{\substack{\ell+\ell'> 0 \\ |\ell'-\ell|<D}} 4^{-(\ell+\ell')} e_{G'}(\cD_\ell,\cD'_{\ell'}) \Big(\frac{\lambda}{mn}\Big)^2
\leq \frac{3\lambda^2 \log d}{2m} \,.\]
\end{list}
Adding together Cases~(i),(ii) implies that when $D \geq 0$
\[\sum_{\ell+\ell'> 0 } 4^{-(\ell+\ell')} e_{G'}(\cD_\ell,\cD'_{\ell'}) \Big(\frac{\lambda}{mn}\Big)^2 \leq \Big(\frac12+\frac32\Big)\frac{\lambda^2\log d}{m} = \frac{2\lambda^2\log d}m\,,\]
 and \eqref{eq-light-pairs-dyadic} now translates this bound to $32\frac{\lambda^2 \log d}m$, confirming the statement of the lemma (with room to spare) for the case $D > 0$ as required.
\end{proof}
Next, we need to address the support of $x$.
A vector $x\in \R^{mn}$ is called \emph{sparse} if it has at most $n/2$ non-zero entries on each fiber, that is, if
\[ \left|\{i\in[n] : x_{ij} \neq 0\}\right| \leq n/2\quad\mbox{ for all }j\in V(G)\,.\]
The next lemma establishes concentration for $R_l(x,y)$ provided that $x$ is sparse.
\begin{lemma}\label{lem-2-fixed-sparse}
Let $x,y\in\R^{mn}$ be two fixed vectors such that $x$ is sparse, $\|x\|\leq 1$ and $\|y\|\leq 1$.
Let $G$ be an $(m,d,\lambda)$-graph for $d\geq 3$ and $\lambda \geq \sqrt{d}$ and let $H$ be a random $n$-lift of $G$.
For any $a \geq 125$,
\[ \P\left(\left|R_l(x,y)-\E [R_l(x,y)]\right| > a \lambda \log d \right) \leq 2d^{-amn/12}\,.\]
\end{lemma}
\begin{proof}
For any $jj'\in E(G)$ and $i\in[n]$, define
\begin{align*}
  X_{jj'}(i) &=  x_{ij} \sum_{i'=1}^n\one_{\{ij\sim i'j'\}}y_{i'j'} \one_{\{|x_{ij}y_{i'j'}|<\frac{\lambda}{mn}\}}\,.
\end{align*}
By this definition,
\begin{equation*}
R_l(x,y) = \sum_{j\sim j'} \sum_i X_{jj'}(i)\,,
\end{equation*}
and we can now expose the values of $X_{jj'}$ sequentially by going over the pairs of fibers $jj'\in E(G)$ one-by-one (in an arbitrary order), and for each such pair revealing the relevant part of the bijection between the fibers. More precisely, when processing a given pair of fibers $j\sim j'$, we proceed as follows:
\begin{enumerate}
  \item Without loss of generality, suppose $|x_{1j}| \geq |x_{2j}| \geq \ldots \geq |x_{nj}|$, and let $q\in[n]$ be the largest index such that $x_{qj} \neq 0$.
  \item Sequentially go over $i=1,\ldots,q$ and expose the neighbor of $ij$ in the bijection between the two fibers, i.e., $J_i\in [n]$ such that $ij \sim J_i j'$, thereby determining $X_{j j'}(i)$.
\end{enumerate}
Crucially, since the vector $x$ is sparse, it contains at most $n/2$ non-zero entries in any given fiber, and so in the above defined process $q \leq n/2$.

Denote by $(\cF_t)$ the filter corresponding to this process (that is, $\cF_t$ is the $\sigma$-algebra generated by the first $t$ exposed edges), and let $(S_t)$ be Doob's martingale corresponding to the function $R_l(x,y)$ with respect to $(\cF_t)$:
\[ S_t = \E\bigg[\sum_{j\sim j'}\sum_i X_{jj'}(i) \given \cF_{t} \bigg] \,.\]
As usual, $S_0 = \E[R_l(x,y)]$ whereas at the end of the process the martingale equals $R_l(x,y)$.

We wish to analyze the increment $S_t - S_{t-1}$. Suppose that in step $t$ we are now exposing an edge between the fibers $j\sim j'$. Clearly,
if $(k,k') \neq (j,j')$ and our process already exposed the edges between the fibers $k \sim k'$, then their contribution is canceled in $S_t-S_{t-1}$. Furthermore, if the edges between $k \sim k'$ are to be exposed in the future, then
\[ \E\bigg[\sum_i X_{kk'}(i)\given  \cF_{t}\bigg] =   \E\bigg[\sum_i X_{kk'}(i)\given  \cF_{t-1}\bigg]\,,\]
since the bijections between distinct pairs of connected fibers are independent. That is to say, $S_t - S_{t-1}$ contains only terms that arise from the effect of the edge exposed at time $t$ on $\E \left[\sum_i X_{jj'}(i)\given \cF_t \right] $.

In light of this, it suffices to treat the case where $j\sim j'$ is the first pair of fibers processed, and the analysis of $S_t - S_{t-1}$ will hold analogous for any other pair.
In what follows, since we are now concentrating solely on the two fibers $j\sim j'$, we omit the subscripts $j,j'$ from $x$ and $y$ to simplify the notation. Similarly, we use the abbreviation
\begin{equation}
  \label{eq-def-L-ii'}
  L_{i,i'} = \one_{\{|x_{ij}y_{i'j'}|<\frac{\lambda}{mn}\}}\,.
\end{equation}
At step $t =1,2,\ldots,q$ we are therefore exposing the match of $x_t$. For simplicity we will analyze $S_1-S_0$ and by merely changing the indices the same argument would carry to all other values of $t$.
Recall that \begin{equation}\label{eq-S0}
S_0 = \frac{1}{n!} \sum_{\pi} \sum_i x_i y_{\pi(i)} L_{i,\pi(i)} = \frac1n \sum_i \sum_{i'} x_i y_{i'} L_{i,i'}
\end{equation}
and that given the event that $x_1$ is matched to some $I \in [n]$ we have
\begin{align}\label{eq-S1}
S_1 &= x_1 y_I L_{1,I} + \frac{1}{(n-1)!} \sum_{\pi: \pi(1)=I} \sum_{i\geq 2} x_i y_{\pi(i)} L_{i,\pi(i)} \nonumber\\
&= x_1 y_I L_{1,I} + \frac1{n-1} \sum_{i\geq 2} \sum_{i' \neq I} x_i y_{i'} L_{i,i'}\,.
\end{align}
 Let $S_1$ and $\tilde{S}_1$ be two possible values after revealing the match of $x_1$, denoting its index by $\pi(1)$ and $\tilde{\pi}(1)$ respectively. We can now couple the distributions over the remaining entries of $\pi$ and $\tilde{\pi}$ via switching $\pi(1),\tilde{\pi}(1)$. That is, if we let $z = \tilde{\pi}^{-1}(\pi(1))$ then $\pi,\tilde{\pi}$ agree everywhere except possibly on $\{1,z\}$ and there we have
\[ \pi(z) = \tilde{\pi}(1)\,,\quad \tilde{\pi}(z) = \pi(1)\,.\]
Clearly, in any pair of coupled $\pi,\tilde{\pi}$ all summands of the form $x_i y_{i'} L_{i,i'}$ in Eq.~\eqref{eq-S1} cancel from $S_1-\tilde{S}_1$ except when
$i\in\{1,z\}$, hence
\begin{align*}
S_1 - \tilde{S}_1  &= x_1 y_{\pi(1)}L_{1,\pi(1)} + \frac{1}{n-1}\sum_{z\neq 1} x_z y_{\tilde{\pi}(1)} L_{z,\tilde{\pi}(1)} \\
                   & - x_1 y_{\tilde{\pi}(1)}L_{1,\tilde{\pi}(1)} - \frac1{n-1} \sum_{z\neq 1} x_z y_{\pi(1)} L_{z,\pi(1)} \,.
\end{align*}
By definition~\eqref{eq-def-L-ii'} each of the above terms $x_i y_{i'} L_{i,i'}$ is at most $\lambda/mn$ in absolute value, thus repeating this argument for any step $t$ gives that with probability $1$,
\begin{equation}
  \label{eq-st-diff-L-inf}
  |S_t - S_{t-1}| \leq \frac{4\lambda}{mn}\quad\mbox{ for $t=1,\ldots,q$}\,.
\end{equation}
Obtaining an $L^2$ bound on the increments $S_t-S_{t-1}$ is slightly more delicate then the above $L^\infty$ bound.
To this end, we will write $S_1-S_0$ explicitly: Recall from~\eqref{eq-S0},\eqref{eq-S1} that $S_0$ averages over permutations $\pi$ on $[n]$ whereas
 $S_1$ averages over all such permutations $\tilde{\pi}$ that have $\tilde{\pi}(1)=I$ for some $I \in [n]$, which is exposed in $\cF_1$
and identifies the match of $x_1$.

In other words, $S_0$ is the mean of sums analogous to $S_1$ with all possible values $\pi(1)\in\{1,\ldots,n\}$ replacing $I$. Each such value has equal probability and the case $\pi(1)=I$ does not contribute to $S_1-S_0$. In the remaining cases we can
go over the possible values of $z=\pi^{-1}(I) \in\{2,\ldots,n\}$ (each with equal probability) and
couple $\pi,\tilde{\pi}$ using the switching that was used to establish the $L^\infty$ bound, letting them agree everywhere except on $\{1,z\}$.
 Altogether we obtain that
\begin{align*}
S_1 - S_0 &= x_1 y_I L_{1,I} -  \frac{1}{n-1}\sum_{z\neq 1} x_z y_I L_{z,I}\\
& - \frac{1}{n}\sum_{i'} x_1 y_{i'} L_{1,i'} + \frac{1}{n(n-1)}\sum_{i'} \sum_{z\neq 1} x_z y_{i'} L_{z,i'} \,.
 \end{align*}
Since the expressions in the last line do not depend on $I$, we conclude that $\var(S_1-S_0)=\var(Z)$
where
\[Z = x_1 y_I L_{1,I} -  \frac{1}{n-1}\sum_{z\neq 1} x_z y_I L_{z,I}\quad\mbox{with $I$ uniform on $[n]$}\,.\]
Estimating $\var(Z)$ requires extra care due to the indicators $L_{i,i'}$. There are two possible cases:
\begin{enumerate}[(i)]
  \item If $L_{1,I}=1$ then by our ordering of the $x_i$'s according to decreasing absolute values we have $L_{z,I}=1$ for all
  $z \geq 1$ and so
  \[ Z = x_1 y_I L_{1,I} -  \frac{1}{n-1}\sum_{z\neq 1} x_z y_I L_{z,I} = \bigg(x_1 - \frac1{n-1}\sum_{z\neq 1} x_z\bigg) y_I\,.\]
  Moreover, $|x_1|$ is at least the average of the $|x_z|$'s for $z > 1$, and so in this case
  \[ Z^2 \leq 4 x_1^2 y_I^2 L_{1,I}\,.\]

  \item Otherwise, $L_{1,I}=0$ and there exists some $T=T(I) > 1$ such that $L_{z,I}=1$ for all $z \geq T$. In this case
  \[ Z = -\frac{1}{n-1} \sum_{z\geq T} x_z y_I\,.\]
  By Cauchy-Schwartz, in this case we thus have
  \[ Z ^2 \leq \frac{1}{n-1} \sum_{z > 1} x_z^2 y_I^2 L_{z,I}\,.\]
\end{enumerate}
Combining the cases, since $I$ is uniform on $[n]$ it now follows that
\[ \var(Z) \leq \E Z^2 \leq \frac4n  \sum_{i'} x_1^2 y_{i'}^2 L_{1,i'} + \frac1{n(n-1)}\sum_{z>1}\sum_{i'} x_z^2 y_{i'}^2 L_{z,i'}\,.\]

Applying the same analysis to a general $t\in[q]$, while assuming without loss of generality
that the remaining unmatched $y_{i'}$'s are $\{y_t,\ldots,y_n\}$, yields
\begin{align}
  \var\left(S_t-S_{t-1}\mid\cF_{t-1}\right) &\leq \frac4{n-t+1} \sum_{i' \geq t} x_t^2 y_{i'}^2 L_{t,i'} \nonumber\\
  &+ \frac1{(n-t+1)(n-t)}\sum_{z>t} \sum_{i' \geq t} x_z^2 y_{i'}^2 L_{z,i'}\,.  \label{eq-Var-St}
\end{align}
At this point our assumption that $q \leq n/2$ due to the fact that $x$ is sparse plays its important role.
For some $z,i'$ consider the total coefficient of $x_z^2 y_{i'}^2 L_{i,i'}$ after summing \eqref{eq-Var-St} over $t =1,\ldots,q$.
The first expression in \eqref{eq-Var-St} contributes at most $4/(n-q+1) \leq 8/n$
whereas the second one adds up to
\[ \sum_{t \leq q} \frac1{(n-t+1)(n-t)} \leq \frac2 n\,.\]
Altogether we conclude that
\[ \sum_{t=1}^q \var\left(S_t-S_{t-1}\mid\cF_{t-1}\right) \leq \frac{10}{n} \sum_{i,i'}x_i^2 y_{i'}^2 L_{i,{i'}}\,\]
and by extending this analysis to all $md/2$ pairs of connected fibers we get
\begin{align}\label{eq-total-var-St}
  \sum_{t} \var(S_t - S_{t-1} \given \cF_{t-1}) &\leq \frac{10}n \sum_{j \sim j'} \sum_{i,i'} x_{ij}^2 y_{i'j'}^2 \one_{\{|x_{ij}y_{i'j'}|<\lambda/mn\}} \nonumber\\
  &\leq   500\frac{\lambda^2 \log d}{mn}\,,
\end{align}
where the last step was by Lemma~\ref{lem-variance-Rl}.
We can now apply the following large deviation inequality for martingales, which is a special case of a result of
Freedman~\cite{Freedman} (see also \cite{McDiarmid} for a variant of this inequality).
\begin{theorem}
Let $(S_i)_{i=0}^n$ be a martingale with respect to a filter $(\cF_i)$ and let $\Delta_i = S_{i}-S_{i-1}$ denote its increments. Suppose that $|\Delta_i| \leq M$ for all $i$
and that $\sum_{i=1}^n \var(\Delta_i \mid\cF_{i-1}) \leq \sigma^2$. Then for any $s > 0$ we have
\[ \P\left(|S_t - S_0| \geq s\mbox{ for some $t\in[n]$}\right)\leq 2\exp\left[-\frac{s^2}{2(\sigma^2+Ms)}\right]\,.\]
\end{theorem}
In our case by \eqref{eq-st-diff-L-inf} and \eqref{eq-total-var-St} we have $M = 4\frac{\lambda}{mn}$ and $\sigma^2 = 500\frac{\lambda^2\log d}{mn}$. Since the final value of $(S_t)$ is $R_l(x,y)$ whereas $S_0 = \E[R_l(x,y)]$, we now get
\begin{align*}
\P\left(\left|R_l(x,y) - \E [R_l(x,y)]\right| \geq s\right) &\leq 2\exp\bigg(-\frac{s^2 mn}{8\lambda(125\lambda\log d + s)}\bigg)\,.\end{align*}
In particular, for $s = a\lambda\log d$ with $a \geq 125$ we have $8\lambda(125\lambda\log d + s) \leq 16\lambda s$ and so
\begin{align*} \P\left(\left|R_l(x,y) - \E [R_l(x,y)]\right| \geq s\right) &\leq
2\exp\left(-\frac{s mn}{16 \lambda}\right) \\ &= 2d^{-a mn/(16 \ln 2)} < 2d^{-amn/12} \,,\end{align*}
as required.
\end{proof}
As a corollary, we can now infer the concentration result of Lemma~\ref{lem-2-fixed-sparse} without requiring that $x,y$ should be sparse.
\begin{corollary}\label{cor-2-fixed}
Let $x,y\in\R^{mn}$ be two fixed vectors with $\|x\|\leq 1$ and $\|y\|\leq 1$.
Let $G$ be an $(m,d,\lambda)$-graph for $d\geq 3$ and $\lambda \geq \sqrt{d}$ and let $H$ be a random $n$-lift of $G$.
 Then for any $a \geq 250$,
\[ \P\left(\left|R_l(x,y)-\E [R_l(x,y)]\right| > a \lambda \log d \right) \leq 2 d^{-amn/24}\,.\]
\end{corollary}
\begin{proof}
Let $x,y\in\R^{mn}$ satisfy $\|x\|\leq 1$ and $\|y\|\leq 1$.
Define the sparse vectors $x',x''\in\R^{mn}$ by
\[ x'_{ij} = \left\{\begin{array}
  {ll} x_{ij} & i \leq n/2 \\
  0 & i > n/2
\end{array}\right.\quad\mbox{and}\quad x''_{ij} = \left\{\begin{array}
  {ll}   0 & i > n/2 \\ x_{ij} & i \leq n/2
\end{array}\right.\,\mbox{ for all $i,j$}\,.\]
Since for every $i,j$ we have that $x_{ij}$ is precisely one of $\{x'_{ij},x''_{ij}\}$ while the other is $0$, we deduce that
\begin{equation}
  \label{eq-Rl-sparse-partition}
  R_l(x,y) = R_l(x',y) + R_l(x'',y)\,.
\end{equation}
Clearly, by the triangle inequality if $|R_l(x,y)-\E R_l(x,y)| \geq a \lambda\log d$ then at least one of the variables $R_l(x',y), R_l(x'',y)$ must deviate from its mean by at least $\frac12 a\lambda\log d$.
For each of the pairs $x',y$ and $x'',y$ we may apply Lemma~\ref{lem-2-fixed-sparse} for a choice of $a/2 \geq 125$ and obtain that
\begin{align*}
\P\left(\left|R_l(x',y)-\E [R_l(x',y)]\right| > (a/2) \lambda \log d \right) &\leq 2d^{-amn/24}\,,
\end{align*}
and the same applies to $x''$. The required result immediately follows.
\end{proof}
To carry the result from the above corollary to every pair of vectors in the lattice $\cL = \big(\frac{1}{d\sqrt{mn}}\Z\big)^{mn}$ we need the following simple claim:
\begin{claim}
There are at most $(4\sqrt{2} d)^{mn}$ vectors $x\in\cL$ such that $\|x\| \leq 1$.
\end{claim}
\begin{proof}
Let $T = \{ x \in \cL : \|x\| \leq 1\}$, set $r = (d+1)/d$ and consider $B_0(r)$, the $mn$-dimensional ball centered at $0$ with radius $r$. For each $x \in T$, define the set
\[Z_x = \{z : x_{ij} < z_{ij} < x_{ij} + 1/(d\sqrt{mn})\}\,.\]
Clearly, each $z \in Z_x$ satisfies $\|z\| \leq \|x\| + \sqrt{mn / (d\sqrt{mn})^2} = (d+1)/d$,
and so $Z_x \subset B_0(r)$. Furthermore, for any $x\neq y\in \cL$ we have $Z_x \cap Z_y =\emptyset$, and
altogether, if we let $\vol(\cdot)$ denote the Lebesgue measure on $\R^{mn}$ then
\[ \sum_{x \in T} \vol(Z_x) \leq \vol(B_0(r)) \leq \frac{(\pi r^2)^{mn/2}}{\lfloor mn/2\rfloor!}
\leq \bigg(\sqrt{2\pi\mathrm{e}}\frac{d+1}{d\sqrt{mn}}\bigg)^{mn}
\,,\]
with the last inequality following from the fact that $k! \geq (k/\mathrm{e})^k$ for all $k$.
Since for every $x$ we have $\vol(Z_x) = (d\sqrt{mn})^{-mn}$, we now deduce that
\[ |T| \leq \left(\sqrt{2\pi\mathrm{e}}(d+1)\right)^{mn}\,.\qedhere\]
\end{proof}
Applying Corollary~\ref{cor-2-fixed} to all $x,y\in T$ with a choice of $a=250$ and then taking a union bound over all $|T|^2$ possible pairs now completes the proof of the Theorem~\ref{thm-light}.
\end{proof}
\section{The second eigenvalue of a random lift}\label{sec:mainproof}

\begin{proof}[\textbf{\emph{Proof of Theorem~\ref{thm-1}}}]
First, note that since $\lambda$ is an upper bound on all nontrivial eigenvalues in absolute value, we can always increase it and take $\lambda \geq \sqrt{d}$, and this would not effect the result (recall that the bound we target for is $(\lambda \vee \rho)\log \rho$ where $\rho=2\sqrt{d-1}$). In this case, as $\rho = 2\sqrt{d-1}$, it suffices to show that every nontrivial eigenvalue of $G$ is $O(\lambda \log d)$ except with probability $O(n^{-100})$.

The following lemma establishes the expected value of $x^\tr A_H\, y$ for any two unit vectors orthogonal to the all-ones vector.
\begin{lemma}\label{lem-expected-x^T-A-y}
  Let $G$ be an $(m,d,\lambda)$-graph and $H$ be a random $n$-lift of $G$.
Let $x,y\in\R^{mn}$ satisfy $\left<x,\allone\right>=\left<y,\allone\right>=0$ and $\|x\|=\|y\|=1$.
Then $\E [x^\tr A_H\, y] \leq \lambda$.
\end{lemma}
\begin{proof}
Clearly, since the $n$-lift is comprised of a uniform perfect matching between any two fibers that are adjacent in $G$, we have
\begin{align*}
\E x^\tr A_H \, y &= \E \sum_{ij \sim i'j'} x_{ij}y_{i'j'} = \frac1n \sum_{j \sim j'} \sum_{i,i'}
x_{ij} y_{i'j'} \\
&= \frac1n \sum_{j \sim j'} \sum_i x_{ij} \sum_{i'} y_{i'j'} =
\frac1n w^\tr A_G \, z\,,
\end{align*}
where $w,z\in\R^m$ are defined by $w_j = \sum_i x_{ij}$ and $z_{j'} = \sum_{i'} y_{i'j'}$ for $j\in[m]$.
The assumptions on $x,y$ give that
\begin{align*} \sum_j w_j = \left<x,\allone\right> = 0\,,\quad&
\sum_{j'} z_{j'} = \left<y,\allone\right> = 0\,,
\end{align*}
and furthermore, by Cauchy-Schwartz,
\[ \|w\|^2  = \sum_j \Big(\sum_i x_{ij}\Big)^2 \leq \sum_j n \sum_i x_{ij}^2 = n \|x\|^2 = n\,,\]
and similarly $\|z\|^2 \leq n$.
Altogether, as $w,z$ are orthogonal to the trivial eigenvector, and since every nontrivial eigenvalue of $A_G$ is at most $\lambda$ in absolute value,
\[ \E x^\tr A_H \, y = \frac1n w^\tr A_G\, z \leq \frac1n \lambda \|w\| \|z\| \leq \lambda\,,  \]
as required.
\end{proof}
To prove Theorem~\ref{thm-1}, assume the events described in Theorem~\ref{thm-heavy} and Theorem~\ref{thm-light} occur.
Let $x,y\in\R^{mn}$ satisfy $\left<x,\allone\right>=\left<y,\allone\right>=0$ and $\|x\|\leq 1$, $\|y\| \leq 1$.
Consider $\tilde{x},\tilde{y}$, the closet vectors to $x,y$ respectively among all vectors in $\{z \in \cL : \|z\|\leq 1\}$,
where $\cL = \big(\frac{1}{d\sqrt{mn}}\Z\big)^{mn}$.
Theorem~\ref{thm-heavy} and Theorem~\ref{thm-light} now give that
\begin{align*}
 | R_h(\tilde{x},\tilde{y}) - \E [R_h(\tilde{x},\tilde{y})] | &\leq 7000 \, \lambda \log d\,,\\
 | R_l(\tilde{x},\tilde{y}) - \E [R_l(\tilde{x},\tilde{y})] | &\leq 250 \, \lambda \log d\,,
\end{align*}
and since $\tilde{x}^\tr A_H\, \tilde{y} = R_l(\tilde{x},\tilde{y}) + R_h(\tilde{x},\tilde{y})$ by definition, we get
\begin{align*}
\left|\tilde{x}^\tr A_H\, \tilde{y} \right| &\leq \left|\tilde{x}^\tr A_H\, \tilde{y} - \E\left[\tilde{x}^\tr A_H\, \tilde{y}\right]\right| + \left|\E\left[\tilde{x}^\tr A_H\, \tilde{y}\right]\right| \\
&\leq \left|R_l(\tilde{x},\tilde{y}) - \E [R_l(\tilde{x},\tilde{y})]\right|
 + \left|R_h(\tilde{x},\tilde{y}) - \E [R_h(\tilde{x},\tilde{y})]\right| + \left|\E\left[\tilde{x}^\tr A_H\, \tilde{y}\right]\right|  \\
 &\leq \lambda+7250 \,\lambda \log d\,.
\end{align*}
Finally, by the definition of the lattice $\cL$, both $x' = x -\tilde{x}$ and $y' = y-\tilde{y}$ satisfy
\[ \|x'\|_\infty \leq \frac{1}{d\sqrt{mn}}\,,\quad \|y'\|_\infty \leq \frac{1}{d\sqrt{mn}}\,,\]
and so $\|x'\| \leq 1/d$ and $\|y'\| \leq 1/d$.
Therefore, for instance,
\[ \left|{x'}^\tr A_H \, \tilde{y}\right| \leq \|x'\| \|A_H\;\tilde{y}\| \leq \frac{1}d\cdot d = 1\,,\]
and similarly\ $ \left|{\tilde{x}}^\tr A_H \, y'\right| \leq 1$ and $ \left|{x'}^\tr A_H \, y'\right| \leq 1/d$.
Combining these inequalities, it now follows that
\begin{align*}
\left|x^\tr A_H\, y \right| &\leq \left|\tilde{x}^\tr A_H\, \tilde{y}\right| +
\left|{x'}^\tr A_H \, \tilde{y}\right|  + \left|{\tilde{x}}^\tr A_H \, y'\right| + \left|{x'}^\tr A_H \, y'\right|
\\
&\leq \left|\tilde{x}^\tr A_H\, \tilde{y}\right| + 3 < 7500 \, \lambda \log d \,,
\end{align*}
completing the proof.
\end{proof}


\begin{bibdiv}
\begin{biblist}

\bib{Alon}{article}{
   author={Alon, N.},
   title={Eigenvalues and expanders},
   journal={Combinatorica},
   volume={6},
   date={1986},
   number={2},
   pages={83--96},
}

\bib{AM}{article}{
   author={Alon, N.},
   author={Milman, V. D.},
   title={$\lambda\sb 1,$ isoperimetric inequalities for graphs, and
   superconcentrators},
   journal={J. Combin. Theory Ser. B},
   volume={38},
   date={1985},
   number={1},
   pages={73--88},
}


\bib{AS}{book}{
  author={Alon, Noga},
  author={Spencer, Joel H.},
  title={The probabilistic method},
  edition={3},
  publisher={John Wiley \& Sons Inc.},
  place={Hoboken, NJ},
  date={2008},
  pages={xviii+352},
}

\bib{AL1}{article}{
   author={Amit, Alon},
   author={Linial, Nathan},
   title={Random graph coverings. I. General theory and graph connectivity},
   journal={Combinatorica},
   volume={22},
   date={2002},
   number={1},
   pages={1--18},
}

\bib{AL2}{article}{
   author={Amit, Alon},
   author={Linial, Nathan},
   title={Random lifts of graphs: edge expansion},
   journal={Combin. Probab. Comput.},
   volume={15},
   date={2006},
   number={3},
   pages={317--332},
}

\bib{BL}{article}{
   author={Bilu, Yonatan},
   author={Linial, Nathan},
   title={Lifts, discrepancy and nearly optimal spectral gap},
   journal={Combinatorica},
   volume={26},
   date={2006},
   number={5},
   pages={495--519},
}

\bib{BS}{article}{
  title={On the second eigenvalue of random regular graphs},
  author={Broder, Andrei},
  author={Shamir, Eli},
  conference={
      title={Proc. of the 28th Annual Symposium on the Foundations of Computer Science},
      year={1987},
  },
  volume={},
  number={},
  pages={286--294},
}


\bib{Dodziuk}{article}{
   author={Dodziuk, Jozef},
   title={Difference equations, isoperimetric inequality and transience of
   certain random walks},
   journal={Trans. Amer. Math. Soc.},
   volume={284},
   date={1984},
   number={2},
   pages={787--794},
}


\bib{HLW}{article}{
  author={Hoory, Shlomo},
  author={Linial, Nathan},
  author={Wigderson, Avi},
  title={Expander graphs and their applications},
  journal={Bull. Amer. Math. Soc.},
  volume={43},
  date={2006},
  number={4},
  pages={439--561},
}

\bib{Freedman}{article}{
   author={Freedman, David A.},
   title={On tail probabilities for martingales},
   journal={Ann. Probability},
   volume={3},
   date={1975},
   pages={100--118},
}

\bib{Friedman08}{article}{
  author={Friedman, Joel},
  title={A proof of Alon's  second eigenvalue conjecture and related problem},
  journal={Mem. Amer. Math. Soc.},
  volume={195},
  date={2008},
  number={910},
}

\bib{Friedman1}{article}{
   author={Friedman, Joel},
   title={Relative expanders or weakly relatively Ramanujan graphs},
   journal={Duke Math. J.},
   volume={118},
   date={2003},
   number={1},
   pages={19--35},
}

\bib{Friedman2}{article}{
   author={Friedman, Joel},
   title={Some geometric aspects of graphs and their eigenfunctions},
   journal={Duke Math. J.},
   volume={69},
   date={1993},
   number={3},
   pages={487--525},
}

\bib{FKS}{article}{
  author    = {Friedman, Joel},
  author    = {Kahn, Jeff},
  author    = {Szemer{\'e}di, Endre},
  title     = {On the Second Eigenvalue in Random Regular Graphs},
  conference = {
    title    = {Annual ACM Symposium on Theory of Computing (STOC 1989)},
  },
  pages     = {587--598},
}

\bib{Greenberg}{article}{
    author = {Greenberg, Y.},
    title  = {On the Spectrum of Graphs and their Universal Coverings},
    journal = {PhD thesis, Hebrew University of Jerusalem},
    year = {1995},
    note = {(in Hebrew)},
}

\bib{LP}{article}{
   author={Linial, Nati},
   author={Puder, Doron},
   title={Word maps and spectra of random graph lifts},
   journal={Random Structures and Algorithms},
   status={to appear},
}

\bib{LS}{article}{
    author = {Lubetzky, Eyal},
    author = {Sly, Allan},
    title = {Cutoff phenomena for random walks on random regular graphs},
    journal={Duke Math. J.},
    status = {to appear},
}

\bib{Lubotzky}{book}{
   author={Lubotzky, Alexander},
   title={Discrete groups, expanding graphs and invariant measures},
   series={Progress in Mathematics},
   volume={125},
   note={With an appendix by Jonathan D. Rogawski},
   publisher={Birkh\"auser Verlag},
   place={Basel},
   date={1994},
   pages={xii+195},
}


\bib{LPS}{article}{
  author={Lubotzky, Alex},
  author={Phillips, Ralph},
  author={Sarnak, Peter},
  title={Ramanujan graphs},
  journal={Combinatorica},
  volume={8},
  date={1988},
  number={3},
  pages={261--277},
}

\bib{Margulis1}{article}{
   author={Margulis, G. A.},
   title={Explicit constructions of expanders},
   language={Russian},
   journal={Problemy Pereda\v ci Informacii},
   volume={9},
   date={1973},
   number={4},
   pages={71--80},
}

\bib{Margulis2}{article}{
   author={Margulis, G. A.},
   title={Explicit group-theoretic constructions of combinatorial schemes
   and their applications in the construction of expanders and
   concentrators},
      journal={Problems Inform. Transmission},
   volume={24},
   date={1988},
   number={1},
      pages={39--46},
}

\bib{McDiarmid}{article}{
   author={McDiarmid, Colin},
   title={Concentration},
   conference={
      title={Probabilistic methods for algorithmic discrete mathematics},
   },
   book={
      series={Algorithms Combin.},
      volume={16},
      publisher={Springer},
      place={Berlin},
   },
   date={1998},
   pages={195--248},
}

\bib{Nilli}{article}{
   author={Nilli, A.},
   title={On the second eigenvalue of a graph},
   journal={Discrete Math.},
   volume={91},
   date={1991},
   number={2},
   pages={207--210},
}

\bib{Pinsker}{article}{
  title={On the complexity of a concentrator},
  author={Pinsker, M.},
  journal = {Proc. of the 7th International Teletraffic Conference},
  date={1973},
  pages={318/1--318/4},
}

\bib{RVW}{article}{
   author={Reingold, Omer},
   author={Vadhan, Salil},
   author={Wigderson, Avi},
   title={Entropy waves, the zig-zag graph product, and new constant-degree
   expanders},
   journal={Ann. of Math.},
   volume={155},
   date={2002},
   number={1},
   pages={157--187},
}

\bib{Sarnak}{article}{
   author={Sarnak, Peter},
   title={What is$\dots$an expander?},
   journal={Notices Amer. Math. Soc.},
   volume={51},
   date={2004},
   number={7},
   pages={762--763},
}


\bib{JS}{article}{
   author={Sinclair, Alistair},
   author={Jerrum, Mark},
   title={Approximate counting, uniform generation and rapidly mixing Markov chains},
   journal={Inform. and Comput.},
   volume={82},
   date={1989},
   number={1},
   pages={93--133},
}

\bib{Wormald}{article}{
   author={Wormald, N. C.},
   title={Models of random regular graphs},
   conference={
      title={Surveys in combinatorics, 1999 (Canterbury)},
   },
   book={
      series={London Math. Soc. Lecture Note Ser.},
      volume={267},
      publisher={Cambridge Univ. Press},
      place={Cambridge},
   },
   date={1999},
   pages={239--298},
}

\end{biblist}
\end{bibdiv}
\end{document}